\documentclass[12pt,a4paper]{article}

\usepackage[T1]{fontenc}
\usepackage[utf8]{inputenc}
\usepackage[english]{babel}
\usepackage{amsfonts, amssymb, amsmath, amsthm}
\usepackage[all]{xy}
\usepackage{enumerate}

\usepackage{graphicx}
\usepackage{psfrag}
\usepackage{geometry}
\usepackage[colorlinks=true,pdfstartview=FitV,linkcolor=blue,citecolor=green,urlcolor=black,filecolor=magenta]{hyperref}
\usepackage{verbatim}

\newtheorem{theorem}{Theorem}[section]
\newtheorem{definition}[theorem]{Definition}
\newtheorem{prop}[theorem]{Proposition}
\newtheorem{lemma}[theorem]{Lemma}
\newtheorem{cor}[theorem]{Corollary}

\newcommand{\Ee}{{\mathcal E}}

\newcommand{\Rr}{{\mathcal R}}

\newcommand{\Uu}{{\mathcal U}}

\newcommand{\CC}{{\bf C}}

\newcommand{\NN}{{\bf N}}

\newcommand{\id}{\mbox{id}}

\newcommand{\AM}{{\mathbb A}}

\newcommand{\TM}{{\mathbb T}}

\newcommand{\R}{\mathbb{R}}
\newcommand{\N}{\mathbb{N}}
\newcommand{\Z}{\mathbb{Z}}
\newcommand{\C}{\mathbb{C}}

\newcommand{\Del}{{\mathcal L}}

\newcommand{\Xmax}{X_{max}}
\newcommand{\XMmax}{\X_{max}}
\newcommand{\pmax}{\pi_{max}}
\newcommand{\fb}[1]{\pmax^{-1}(#1)}
\newcommand{\pt}{I}
\newcommand{\pct}{\mathfrak p}
\newcommand{\tct}{\mathfrak t}
\newcommand{\Tct}{\mathfrak T}
\newcommand{\nd}{{N}}
\newcommand{\np}{{N^\perp}}
\newcommand{\Omax}{\Omega_{max}}
\newcommand{\X}{{\Xi}}
\newcommand{\Tp}{\TM^\perp}
\newcommand{\OP}{\Omega_\Phi}
\newcommand{\Phimax}{\Phi_{max}}
\newcommand{\supp}{\mbox{\rm supp}}

\newcommand{\TT}{\mathbb{T}}

\newcommand{\Gdual}{\widehat{{G}}}
\newcommand{\Oomega}{(X,G)}

\newcommand{\Ttheta}{(Y,G)}

\newcommand{\oplam}{\mbox{\Large $\curlywedge$}}

\newcommand{\Hm}[1]{\leavevmode{\marginpar{\tiny%
$\hbox to 0mm{\hspace*{-0.5mm}$\leftarrow$\hss}%
\vcenter{\vrule depth 0.1mm height 0.1mm width \the\marginparwidth}%
\hbox to 0mm{\hss$\rightarrow$\hspace*{-0.5mm}}$\\\relax\raggedright
#1}}}

\title{Equicontinuous factors, proximality and Ellis semigroup for
  Delone sets}
\author{J.-B.~Aujogue, M.~Barge, J.~Kellendonk, D.~Lenz}
\date{\today}

\begin{document}

\maketitle

\tableofcontents

\section{Introduction}

Recurrence properties of patterns in a Euclidean point set may be effectively studied by means of an associated dynamical system. The idea, due originally to Dan Rudolph (\cite{Rudolph}) and analogous to the notion of a sub-shift in Symbolic Dynamics, is as follows: Given a point set $\Del\subset\R^N$, let $\Omega_\Del$, the {\em hull} of $\Del$, be the collection of all point sets in $\R^N$, each of which is locally indistinguishable from $\Del$. There is a natural {\em local topology} on $\Omega_\Del$, $\R^N$ acts on $\Omega_\Del$ by translation, and the structure of $\Del$ is encoded in the topological dynamics of the system $(\Omega_\Del,\R^N)$. In this article we consider the correspondence between  properties of the point set $\Del$ and the closeness to equicontinuity of  the system $(\Omega_\Del,\R^N)$. (For an excellent survey of the dynamical properties of the hull of a point set, with an emphasis on sets arising from substitutions, see \cite{Robinson2004}.)

Consider, for example, the `crystalline' case of a completely periodic set $\Del=\{a+kv_i:i=1,\ldots,N,\, k\in\Z\}\subset\R^N$ with $a\in\R^N$ and linearly independent $v_i\in\R^N$. The hull of $\Del$ is in this case just the set of translates $\Omega_\Del=\{\Del-t:t\in\R^N\}$. Letting $L$ be the lattice spanned by the basis $\{v_i\}_{i=1}^N$, we may identify $\Omega_\Del$ with the torus $\R^N/L$ by $\Del-t\leftrightarrow L-t$, and we see that the dynamical system $(\Omega_\Del,\R^N)$ is simply the translation action on a compact abelian group. This is an {\em equicontinuous action}, that is, the collection of homeomorphisms $\{\alpha_t:\Del'\mapsto \Del'-t\}_{t\in\R^n}$ of $\Omega_\Del$ is an equicontinuous family; in particular, if $\Del',\Del''\in \Omega_\Del$ are {\em proximal} (that is, the distance between $\Del'-t$ and $\Del''-t$ is not bounded away from zero)  then $\Del'=\Del''$.

We are of course more interested in $\Del$ that are highly structured (in this article, always Delone with finite local complexity, usually repetitive, and often Meyer) but not periodic.
The hull will no longer be a group, and there will be distinct elements which are proximal so the action is no longer equicontinuous, but we can
begin to understand the system $(\Omega_\Del,\R^N)$, and thus the structure of $\Del$, by comparing $(\Omega_\Del,\R^N)$ with its largest (maximal) equicontinuous factor $(\Omax,\R^N)$.
Existence of a maximal equicontinuous factor is immediate: Let $R_{max}$ be the intersection of all closed invariant equivalence relations  $R$ on $\Omega_\Del$ for which the action on $\Omega_\Del/R$ is equicontinuous. Then $\Omax:=\Omega_\Del/R_{max}$ is a maximal equicontinuous factor; uniqueness of $\Omax$ (up to topological conjugacy) is a consequence of maximality. But understanding what is collapsed in passing from $\Omega_\Del$ to $\Omax$ requires a more concrete formulation of the equicontinuous structure relation $R_{max}$.
It is clear that elements $\Del',\Del''\in \Omega_\Del$ which are proximal must be identified in $\Omax$, and for many of the familiar $\Del$, $R_{max}$ is just the proximal relation. Clearly, the equicontinuous structure relation is a closed equivalence relation; in general, however, the proximal relation is neither closed nor transitive. To illustrate this, consider the example\footnote{This example is not repetitive and hence is a bit misleading in its simplicity: the hull of any repetitive and non-periodic Delone set is considerably more complicated, having uncoutably many $\R^N$-orbits and being locally the product of a Cantor set with a Euclidean disk (see, for example, \cite{SadunWilliams}).}  with $\Del:=-2\N\cup\N\subset\R$. The hull $\Omega_\Del$ is the disjoint union of three $\R$-orbits, namely it contains besides the translates of $\Del$ also those of $\Del^+:=\Z$ and $\Del^-:=2\Z$.
It thus consists of two circles $\{\Del^{+}-t: t\in\R\}=\R/\Z$, $\{\Del^{-}-t: t\in\R\}=\R/2\Z$ and
the curve $\{\Del-t:t\in\R\}$ winding from one to the other. Notice that the distance between $\Del-t$ and  $\Del^{+}-t$ goes to zero as $t\to +\infty$
and the distance between $\Del-t$ and  $\Del^{-}-t$ goes to zero as $t\to -\infty$.
Thus $\Del$ is proximal with each of $\Del^-$ and $\Del^+$, but these latter are not proximal with each other. So the proximal relation is not transitive.
Note also that, for $m\in\Z$, $\Del^-=\Del^--2m$ is proximal with $\Del-2m$ and the latter tends to $\Del^+$ as $m\to\infty$. Thus the proximal relation is not topologically closed. The maximal equicontinuous factor is the circle $\Omax=\R/\Z$ with $\infty$-to-1 factor map $\pi_{max}:\Omega_\Del
 \to\Omax$ given by $\pi(\Del'-t)=\Z-t$ for $\Del'\in\{\Del^-,\Del,\Del^+\}$.

 In Section \ref{Equicontinuous} we show that points of $\Omega_\Del$ are identified in $\Omax$
if and only if they have the same image under all continuous eigenfunctions of $(\Omega_\Del,\R^N)$ (see Theorem \ref{conjugate})  and in Section \ref{Proximality}, several intrinsically defined variants of the proximal relation are discussed. These include a uniform version of proximality called {\em syndetic proximality} and  {\em regional proximality}, which for minimal systems always agrees with the equicontinuous structure relation. Each of the proximal, syndetic proximal and regional proximal relations takes a `strong' form for hulls of repetitive Meyer sets (repetitive Delone sets $\Del$ for which $\Del-\Del$ is also Delone). For hulls of repetitive Delone sets, if the regional proximal relation is equal to a statistical version of proximality called {\em statistical coincidence}, all of these relations are the same and the system is close to being equicontinuous in a specified sense (Corollary~\ref{Agreement}).

To measure the closeness of the system $(\Omega_\Del,\R^N)$ to equicontinuity, one can consider the cardinality of a fiber $\pi_{max}^{-1}(\xi)$, $\xi\in\Omax$, of the equicontinuous factor map $\pi_{max}$. But which $\xi$? By ergodicity of $(\Omax,\R^N)$ there is $m\in\N\cup\{\infty\}$ such that $\sharp\pi_{max}^{-1}(\xi)=m$ for Haar-almost all $\xi$, but there are other useful notions for the rank of $\pi_{max}$. Two obvious ones are the {\em minimal rank} and {\em maximal rank}: $mr:=\inf_{\xi\in\Omax}\sharp\pi_{max}^{-1}(\xi)$ and $Mr:=
\sup_{\xi\in\Omax}\sharp\pi_{max}^{-1}(\xi)$. A third, which turns out to be extremely useful,
is the {\em coincidence rank}, $cr$, defined as the supremum, over $\xi\in\Omax$, of the supremum of cardinalities of subsets of $\pi_{max}^{-1}(\xi)$ whose elements are pairwise non-proximal. For a repetitive Delone set $\Del$ with hull $\Omega_\Del$, the proximal relation
and the equicontinuous structure relation $R_{max}$ are the same if and only if $cr=1$ (Theorem \ref{characterization-PequalQ}). If $mr<\infty$, then the proximality relation and $R_{max}$ are the same if and only if the proximality relation is closed (Theorem 2.15 of \cite{BargeKellendonk12}) and if $cr<\infty$, then the proximal relation is closed if and only if $cr=1$ (Theorem \ref{Pclosed}). The meaning of the coincidence rank is revealed most clearly when $\Omega_\Del$ is the hull of a repetitive Meyer set and the fiber distal points have full Haar measure: 
For $R$ sufficiently large, for each $\xi\in\Omax$,
there is a set $A\subset\R^N$ of density 1 so that there are exactly $cr$ distinct sets of the form $\Del'\cap B_R(v)$, $\Del'\in \pi_{max}^{-1}(\xi)$ , for each $v\in A$ (see Theorem~\ref{thm-cr-stat}). That is to say, viewed out to radius $R$, a fiber typically appears to have cardinality $cr$.

Due to its connection with pure point diffraction spectrum of $\Del$ (see \cite{Dwo} and \cite{LMS,BaakeLenz04}), the $mr=1$ case (that is, $\pi_{max}$ is somewhere 1-1) has received the most attention. The following results (the first two due to Baake, Lenz, and Moody \cite{BaakeLenzMoody07} and the third to Aujogue \cite{AujoguePhD}) for the hull of a repetitive Meyer set $\Del$ are discussed in Section \ref{Hierarchy}:
\begin{itemize}
\item $\pi_{max}$ is everywhere 1-1 if and only if $\Del$ is completely periodic (Theorem \ref{Meyer-equal-to-mef});
\item $\pi_{max}$ is almost everywhere 1-1 if and only if $\Del$ is a regular complete Meyer set (Theorem \ref{Meyer-equal-to-mef-almost-everywhere}); and
\item $\pi_{max}$ is somewhere 1-1 if and only if $\Del$ is a complete Meyer set (Theorem \ref{Meyer-equal-to-mef-one-point}).
\end{itemize}
In the above,  a {\em complete Meyer set} is a repetitive inter-model set whose window is the closure of its interior; such a set is  {\em regular} if the boundary of the window has zero measure. This hierarchy gives a rather satisfying picture of the correspondence between injectivity properties of $\pi_{max}$ and structural properties of the set $\Del$.

When the coincidence rank is greater than 1, the equicontinuous structure relation is no longer given by proximality and the situation is considerably more complicated. We are able to make a few observations in Section \ref{Proximality} when the coincidence rank is known to be finite (as is the case for `Pisot type' substitutive systems). For example, if $cr<\infty$ for the dynamical system on the hull of an $N$-dimensional repetitive Delone set with finite local complexity, then the system is topologically conjugate to that on the hull of a repetitive Meyer set. Moreover, if the Delone set has no periods, then its topological eigenvalues are dense in $\hat{\R}^N$ (Theorem \ref{thm-cr-finite}). Also, for repetitive Delone sets whose system $(\Omega_\Del,\R^N)$ has finite coincidence rank, if the set of points $\xi$ in $\Omax$ with the property that the points in the fiber $\pi_{max}^{-1}(\xi)$ are pairwise non proximal has full Haar measure, and if $\mu$ is any ergodic probability measure on $\Omega_\Del$, then the continuous eigenfunctions generate $L^2(\Omega_\Del,\mu)$ if and only if $cr=1$, and these conditions imply unique ergodicity of $(\Omega_\Del,\R^N)$ (see Theorem \ref{characterisation-pure-point}). But a systematic understanding of the structure of $(\Omega_\Del,\R^N)$ when the minimal rank is greater than 1 remains a challenging problem for the future.

One can view the hull of $\Del$ and its maximal equicontinuous factor as compactifications of the acting group $\R^N$ (at least
when $\Del$ has no periods, so that $\R^N$ acts faithfully). Another compactification, which preserves more of the topology of $\Omega_\Del$ than does $\Omax$, while still introducing additional algebraic structure, is provided by the Ellis semigroup $E(\Omega_\Del,\R^N)$. This is defined as the closure of the set of homeomorphisms $\{\alpha_t:\Del'\mapsto \Del'-t\,|\,t\in\R^n\}\subset \Omega_\Del^{\Omega_\Del}$ in the Tychonov topology with semigroup operation given by composition. For our considerations, two algebraic properties of the Ellis semigroup are
particularly relevant: $E(\Omega_\Del,\R^N)$ has a unique minimal ideal if and only if the proximal relation is an equivalence relation; and two elements of $\Omega_\Del$ are proximal if and only if they have the same image under some idempotent belonging to a minimal ideal of $E(\Omega_\Del,\R^N)$.
Consider the example above with $\Del:=-2\N\cup\N$. The sequence $(\alpha_{2m})_{m\in\N}$ converges to the element $\alpha^-\in E(\Omega_\Del,\R)$: $\alpha^-$ is the identity on each of the circles $\TT^{\pm}:=\{\Del^{\pm}-t:t\in\R\}$ and collapses the curve $\{\Del-t:t\in\R\}$ onto $\TT^-$ by $\Del-t\mapsto \Del^--t$. In particular, $\alpha^-$ is an idempotent ($\alpha^-\circ \alpha^-=\alpha^-$) and $\alpha^-$ identifies the proximal points $\Del$ and $\Del^-$.
There is a similarly defined idempotent $\alpha^+$ that identifies $\Del$ and $\Del^+$, and the Ellis semigroup is isomorphic with the disjoint union: $E(\Omega_\Del,\R)\simeq \TT\times\{\alpha^-\}\cup \R\times\{id\}\cup\TT\times\{\alpha^+\}$, where $\TT=\Omax$ which is a group, and the operation is coordinate-wise (and non-abelian, since $\alpha^-\circ \alpha^+=\alpha^-$ while $\alpha^+\circ \alpha^-=\alpha^+$).
There are two minimal (left) ideals, corresponding to $\TT\times\{\alpha^+\}$ and $\TT\times\{\alpha^-\}$, reflecting the fact that proximality is not transitive.

The Ellis semigroup is typically a very complicated gadget. One can't even expect  countable neighborhood bases for the topology. It is thus surprising that there is a relatively simple algebraic and topological description of the Ellis semigroup of the hull of any {\em almost canonical} cut-and-project set, similar to that given in the previous paragraph. In this description, found by Aujogue in \cite{AujoguePhD} and explained here in Section 5 (see also \cite{Aujogue13}), the finite sub-monoid of idempotents effectively captures the proximal structure of the hull. This family of examples suggests that the deep and well-developed abstract theory of the Ellis semigroup may find significant further application in the study of highly structured Delone sets.

We begin this article with a review of relevant facts from dynamics and basic constructions of Delone sets in Section 2. Maximal equicontinuous factors associated with Delone sets are considered in Section \ref{Equicontinuous} and variants of the proximal relation are discussed in Section \ref{Proximality}. The final Section \ref{Ellis} introduces the Ellis semigroup and concludes with the example of the semigroup of the hull of the Octagonal tiling.

\section{Background}

\subsection{General background}
In this article we  consider  actions of topological groups on topological spaces.
The spaces,  often denoted by $X$,  will be  compact  Hausdorff spaces.
The groups will be locally compact, $\sigma$-compact  \textbf{abelian} groups and mostly  denoted by $G$. The group composition
will  generally be written additively as $+$.  The neutral element  will be denoted by $e$. The dual of a group $G$ consists of all continuous homomorphisms from $G$ to the unit circle. It is again a locally compact abelian group and will be denoted by $\Gdual$. As the elements of $\Gdual$ are maps on $G$ there is a dual pairing between a group $G$  and its  dual group $\Gdual$. It   will be denoted as $(\cdot, \cdot)$.

We will assume metrizability of $X$ in some  cases in order to ease the presentation of certain concepts. If the corresponding results are valid without the metrizability assumption we have stated them without this assumption. In certain cases we will also  need the groups to be compactly generated.  The main application we have in mind are Delone dynamical systems in Euclidean space. These systems are metrizable and the underlying group (Euclidean space) is compactly generated. Thus, \textbf{all} our results below apply in this situation.

\subsection{Background on dynamical systems}
Whenever the locally compact abelian group  $G$
acts on the compact space $X$  by a continuous action
\begin{equation*}
 \alpha : \; G\times X \; \longrightarrow \; X
   \, , \quad (t,x) \, \mapsto \,  \alpha_t x\, ,
\end{equation*}
where $G\times X$ carries the product topology, the triple $(X,G,\alpha)$
 is called a {\em topological dynamical system\/} over $G$. We will mostly suppress  the action $\alpha$ in our notation and write
$$ t\cdot x := \alpha_t x.$$
Accordingly, we will then also suppress $\alpha$ in the notation for the dynamical system and just write $\Oomega$ instead of $(X,G,\alpha)$.

A dynamical system  $\Oomega$ is called {\em
minimal\/} if, for all $x\in X$, the $G$-\textit{orbit}
$\{t\cdot x: t \in G\}$ is dense in $X$.

  Let two topological dynamical systems\/ $\Oomega$ and\/ $\Ttheta$
  over $G$ be given.   Then, a continuous map $\varrho : X\longrightarrow Y$ is called a \textit{$G$-map} if $\varrho(t\cdot x ) = t \cdot \varrho (x)$ holds for all $x\in X$ and $t\in G$. A $G$-map is called a \textit{factor map} if it is onto. In this case,  $\Ttheta$ is called a \textit{factor}  of $\Oomega$. Factor maps will mostly denoted by $\pi$.
A $G$-map is called a \textit{conjugacy} if it  is a homeomorphism. Then, the dynamical systems are called {\em conjugate}.  In this case, of course, each system is a factor of the other.

  \medskip

   An important role in our subsequent considerations will be played by dynamical systems in which $X$ is a compact group and the action is induced by a homomorphism.  In order to simplify the notation we provide a special name for such systems.

\begin{definition} (Rotation on a compact abelian group) A dynamical system $(X,G)$ is called a
rotation on a compact abelian group
if $X$ is a compact abelian group and the action of $G$ on $X$  is induced by a homomorphism $j : G\longrightarrow X$  such that $t \cdot x = j (t) +  x$ for all $t\in G$ and $ x\in X$.
\end{definition}

\textbf{Remarks.}
 \begin{itemize}
 \item The name of rotation on a compact group comes from the example of a rotation by the angle $\alpha \in \R$ on the unit circle. In that case $X$ is given by the unit circle $S^1 =\{ z\in \C : |z| =1\} $ and $j$ is the  map from the group $\Z$ of integers into $S^1$ mapping $n$ to $ e^{i \alpha n}$.

\item  If $G = \R$ such systems are also known as Kronecker flows. More generally, for $G$ arbitrary they are known as Kronecker systems.

\item We will think of the compact group as a form of torus. Accordingly, in the sequel we will often  denote rotations on a compact group by $(\TT,G)$.

\item A rotation on a compact group is minimal if and only if  $j$ has  dense range.
\end{itemize}

\smallskip

Whenever $\Oomega$ is a dynamical system then
 $\chi\in \Gdual$  is called a {\em continuous eigenvalue}  if there exists a continuous $f$ on $X$ with $f\neq 0$ and
$$ f(t\cdot x) = (\chi,t) f(x)$$
for all $t\in G$ and $x\in X$. Such an $f$ is then called an \textit{continuous eigenfunction (to the eigenvalue $\chi$)}.

 If $(X,G)$ is minimal short arguments show the following: Firstly,  any  continuous eigenfunction has constant modulus and therefore  does not vanish anywhere.  Secondly,  two continuous eigenfunction to the same eigenvalue are linear dependent. In particular, the dimension of the space of all continuous eigenfunctions to a fixed eigenvalue is always one.

The set of all continuous eigenvalues is a subgroup of $\Gdual$. Indeed, the constant function $1$ is always a  continuous eigenfunction to the eigenvalue $0$, the product of two continuous eigenfunction is a continuous eigenfunction (to the product of the eigenvalues) and the complex conjugate of a continuous eigenfunction is a continuous eigenfunction (to the inverse of the eigenvalue).
 The group of continuous eigenvalues of the dynamical system $(X,G)$, equipped with the discrete topology, plays a crucial role in the subsequent considerations and we denote it by $\mathcal{E}_{top} (X,G)$.


There is a strong connection between the group of continuous eigenvalues and a certain rotation on a compact group, called the \textit{maximal equicontinuous factor}.  This factor is at the heart of the investigations of this chapter.  In fact, Section \ref{Hierarchy}   is  devoted to  how the maximal equicontinuous factor  controls the original dynamical system and  Section \ref{Proximality} deals with the fine analysis of the equivalence relation induced by this factor.

\smallskip

For our investigations to be meaningful we will need non-triviality of the group of continuous eigenvalues. Dynamical systems without non-trivial continuous eigenvalues are called \textit{topologically weakly mixing}. This property can be seen to be equivalent to transitivity  of the product system with the diagonal action.

\bigskip

In the present description of dynamical systems we have so far been concerned with the topological point of view. Indeed, this is the main focus of our considerations. However, for certain issues we will need measure theoretical aspects as well.

As a compact space the set $X$ underlying the dynamical system $(X,G)$ carries naturally the Borel-$\sigma$-algebra which is the smallest $\sigma$-algebra containing all compact sets.
A measure $m$ on $X$ is  called \textit{invariant} if $m (\alpha_t B) = m(B)$ for any Borel measurable set $B$ and any $t\in G$.
An invariant probability measure is a called \textit{ergodic}  if  any Borel set $B$ with  $\alpha_t (B) = B$ for all $t\in G$ satisfies $m(B) = 0$ or $m(B) =1$. Any dynamical system admits invariant probability measures and the  ergodic measures can be shown to be the extremal points of the set of invariant probability measures (see, for example, the monographs \cite{DenkerGrillenbergerSiegmund76,Walters82}). In particular, if  a dynamical system admits only one invariant probability measure, then this measure is automatically ergodic. In  this case the dynamical system is called \textit{uniquely ergodic}.

Whenever $m$ is an invariant probability  measure on $(X,G)$ the triple $(X,G,m)$ is called a measurable dynamical system. The   action of $G$  induces a unitary representation on $L^2 (X,m)$ as follows: For any $t\in G$ there is a unitary map
$$T_t : L^2 (X,m)\longrightarrow L^2 (X,m), \: T_t f (x) = f(t \cdot x).$$
The behaviour of the action can be analysed through the behaviour of the representation $T$. This is sometimes known as \textit{Koopmanism}.

As usual an element $f\in L^2 (X,m)$ is called a \textit{measurable eigenfunction} to the  \textit{measurable eigenvalue} $\chi \in \Gdual$ if $$ T_t f  = (\chi,t) f$$ holds for all $t\in G$. Here, the equality is meant in the sense of $L^2$.

Ergodicity implies that  the modulus of any measurable  eigenfunction is constant almost surely and that two measurable eigenfunctions to the same eigenvalue are linearly dependent.  The set of all measurable  eigenvalues forms a subgroup of $\Gdual$. This subgroup, equipped with the discrete topology, is  denoted by $\mathcal{E}_{meas} (X,T)$. If  the closed subspace of $L^2 (X,m)$ generated by the eigenfunctions  agrees with $L^2 (X,m)$ then $(X,G,m)$ is said to have \textit{pure point spectrum}.

\subsection{Background on Delone sets}
Let $G$ be a locally compact abelian group.  We will deal with subsets $\Del$ of $G$.
A subset $\Del$ of $G$ is  called \textit{uniformly discrete} if there exists an open neighborhood $U$ of the identity  in $G$ such that
$$(x + U) \cap (y + U) = \emptyset$$
for all $x,y\in \Del$ with $x\neq y$. A subset $\Del$ of $G$ is called \textit{relatively dense} if there exists a compact neighborhood $K$ of the identity of $G$ such that
$$ G = \bigcup_{x\in \Del} (x + K).$$
A subset $\Del$ of $G$ is called a \textit{Delone set} if it is both uniformly discrete and relatively dense. There is a natural action of $G$ on the set $\mathcal{U} (G)$ of   uniformly discrete sets in $G$ via
$$ G\times \mathcal{U}(G) \longrightarrow \mathcal{U} (G), \, (t,\Del)\mapsto \Del -t:=\{x -t  : x\in \Del\}.$$ We refer to it as \textit{translation action}. 

\smallskip

Most prominent is the case $G = \R^N$. In that case one  can express the above definitions using balls with respect to   Euclidean distance.  The open set $U$ is expressed by an open ball and the compact set $K$ by a closed ball.  We denote the open ball with radius $r$ around $x\in \R^N$ by $U_r (x)$ and the closed ball around $x$ with radius $R$ by $B_R (x)$. Then, $\Del \subset \R^N$ is uniformly discrete if and only if  there exists an $r>0$ with $U_r (x) \cap U_r (y) = \emptyset$ for all $x,y\in \Del$ with $x\neq y$, i.e.,  if and only if there exists an $r>0$ such that the distance between any two different points of $\Del$ is at least $2 r$.  Such a set will then be called \textit{$r$-discrete}. The set $\Del \subset \R^N$ is relatively dense if and only if there exists an $R>0$ with $\R^N = \cup_{x\in \Del} B_R (x)$, that is, if and only if any point of $\R^N$ has distance not exceeding $R$ to $\Del$.

\smallskip

Whenever $\Del$ is a uniformly discrete subset of $G$ a set of the form $(\Del - x)\cap K$ with $x\in \Del$ and $K$ compact is  called a \textit{patch of $\Del$}.  A uniformly discrete subset    $\Del$  in $G$ is said to have \textit{finite local complexity} (FLC) if for any compact $K$ in $G$   the set
$$\{ (\Del - x) \cap K : x\in \Del\}$$
is finite.  This just means that there are only finitely many patches for fixed 'size' $K$. It is not hard to see that $\Del$ has finite local complexity if and only if the set
$$\Del - \Del = \{x -y : x,y\in \Del\}$$
is locally finite, i.e., has finite intersection with any compact subset of $G$. This in turn is equivalent to $\Del - \Del$ being closed and discrete. A Delone set with finite local complexity will be referred to as an \textit{FLC Delone set}.

For an  FLC Delone set $\Del$ we define the \textit{hull} $\Omega_\Del$ to be the set  of all Delone sets whose patches are also patches of $\Del$. This set is obviously invariant under the translation action of $G$ given above. Moreover,  it is compact in a natural topology (discussed below). So, when equipped with the translation action, $\Omega_\Del$ becomes a dynamical system, $(\Omega_\Del,G)$, which we refer to as the \textit{dynamical system associated to $\Del$}.

When $G =\R^N$ it is possible to further characterize finite local complexity. A Delone set $\Del\subset\R^N$ with $\R^N = \cup_{x\in \Del}  B_R (x)$ for some $R>0$ has finite local complexity if and only if the set
$$\{ (\Del - x)\cap B_{2R} (0) : x\in \Del\}$$
is finite \cite{Lagarias99}.  Thus, in this case one needs to test for  finiteness of the number of patches only for patches of a  certain fixed size. For this reason, the hull of a Delone  set with finite local complexity in $\R^N$ can be thought of as a geometric analogue to  a subshift over a finite alphabet.

An \textit{occurrence of the patch } $(\Del - x) \cap K$ in a Delone set  $\Del$ is an element of
$$ \{y \in \Del : (\Del - x)\cap K \subset (\Del - y)  \}.$$
A Delone set $\Del$ is called  \textit{repetitive}  if for any nonempty patch  the set of occurrences is relatively dense.
For an FLC Delone set $\Del$, repetitivity is equivalent to minimality of the associated system $(\Omega_\Del, G)$.
\smallskip

There are various equivalent approaches to define a topology on the set of all uniformly discrete sets.
One is based on the identification of point sets with measures and the vague topology
\cite{BellissardHerrmannZarrouati00,BaakeLenz04}, another uses uniform structures \cite{BaakeLenz04}. If $G=\R^N$ one can also make precise the idea of defining a metric by the principle that two sets are the $\epsilon$-close if they coincide up to an error of $\epsilon$ on the $\frac1\epsilon$-ball around $0$ \cite{Rudolph,Solomyak,FHKmem}.
In general the error of coincidence is measured with the help of the Hausdorff distance. A particularily elegant formulation of this idea uses the stereographic projection and has been worked out in detail in \cite{LenzStollmann01}.

If the sets in question have finite local complexity, which is the only case which we consider in more detail here, then the description of the topology simplifies.
A net
\footnote{A \textit{net} in a topological space $X$ is a function from a directed set $A$ to $X$: the image of $\alpha\in A$ is denoted $x_{\alpha}$ and the net is denoted $(x_{\alpha})_{\alpha}$. The net \textit{converges} to $x\in X$ if for each neighborhood $U$ of $x$ there is $\beta\in A$ so that $x_{\alpha}\in U$ for all $\alpha\ge\beta$. Convergence of nets completely describes the topology of $X$ - see \cite{Kel}.} $(\Del_\iota)_\iota$ in the hull $\Omega_\Del$ of an FLC Delone set
converges to $\Del'$ if and only if there exists a net $(t_\iota)_\iota$ in $G$ converging to $e$ and for all compact $K\subset G$ an $\iota_K$ such that $(\Del_\iota -t_\iota) \cap K = \Del^{'} \cap K$ for  all $\iota > \iota_K$.
Moreover, if $G=\R^N$ then the topology on hull $\Omega_\Del$ of a Delone set with finite local complexity is induced by the metric \cite{AP}
$$d(\Del,\Del') = \inf \left\{\frac\epsilon{\epsilon+1} : \exists t,t'\in B_{\epsilon}(0): B_{\frac1\epsilon}[\Del-t]=B_\frac1\epsilon[\Del'-t']\right\}.$$
Here $B_R[\Del]=\Del\cap B_R(0)$ is the $R$-patch of $\Del$ at $0$, i.e., the patch defined by the closed $R$-ball at $0$.


\subsection{Background on lattices, Model sets and  Meyer sets}\label{Background-Meyer}
Meyer sets  can be thought of as (quite natural) generalizations of lattices.
They have been introduced by Meyer in \cite{Meyer72} in the purely theoretical context of `expanding sets via Fourier transforms'. After the discovery of quasicrystals  they have become a most  prominent class of examples for such structures.   Our discussion follows  \cite{Moody97,Moody00,Schlottmann00} to which we refer for further details and references. For the topic of regular complete Meyer sets we also highlight \cite{Moody}, which gives an introduction into the topic by surveying the results of \cite{BaakeLenzMoody07}. 

\bigskip

A \textit{lattice}  in a group $G$ is a uniformly discrete subgroup such that the quotient of $G$ by this subgroup  is compact. Thus,  any lattice is a Delone set.  It is not hard to see that a Delone set $L$ is a lattice if and only if it  satisfies
$$ L - L =  L.$$
 Obviously, this gives that $L - L$ is uniformly discrete and  hence locally finite.   Thus, a lattice is a Delone set of finite local complexity. Moreover, whenever $L$ is a lattice then
$$L^\ast:=\{ \chi \in \Gdual : (\chi,x)  =1 \;\mbox{for all} \; x\in L\}$$
is a lattice in $\Gdual$. It is called the \textit{dual lattice}. Thus, whenever $L$ is a lattice the set of its $\varepsilon$-dual characters
$$L^\varepsilon := \{\chi \in \Gdual : |(\chi,x) - 1| \leq \varepsilon \;\mbox{for all} \; x\in L\}$$
is relatively dense.  Meyer sets are generalizations of lattices. It turns out that they can be characterized by suitable relaxations of each of the features discussed so far. In fact, each of the features given in the next theorem can be seen as a weakening of a corresponding feature of a lattice.

\begin{theorem} \label{theorem-characterization-Meyer} Let $\Del$ be a Delone set in $G$. Then, the following assertions are equivalent:

\begin{itemize}

\item[{\rm (i)}] $\Del - \Del\subset \Del + F$ for some finite set $F\subset G$.

\item[{\rm (ii)}] For any $\varepsilon >0$ the set
$$\Del^\varepsilon =\{ \chi \in \Gdual : |(\chi,x) - 1| \leq \varepsilon \;\mbox{for all} \; x\in \Del\}$$
of $\varepsilon$-dual characters of $\Del$ is relatively dense in $\Gdual$.
\item[{\rm (iii)}] There exists a cut-and-project scheme $(H,\tilde{L})$ over $G$ together with a compact $W$ with $\Del \subset \oplam (W)$.
\end{itemize}
\end{theorem}

 Here, the last point requires some explanation.
 A {\em cut-and-project scheme\/ over $G$} consists of a locally compact abelian group $H$ together with a  lattice $\widetilde{L}$ in $G
 \times H$ such that the two natural projections $p_1 \! : \, G\times H
 \longrightarrow G$, $(t,h)\mapsto t,$ and $p_2 \! : \, G\times H
 \longrightarrow H$, $(t,h)\mapsto h,$
satisfy the following properties:
\begin{itemize}
\item The restriction $p_1|_{\tilde{L}}$ of $p_1$ to
      $\tilde{L}$ is injective.
\item The image $p_2(\tilde{L})$ is dense in $H$.
\end{itemize}
Let $L:=p_1 (\tilde{L})$ and $(.)^\star \! : \, L \longrightarrow H$ be the
mapping $p_2 \circ (p_1|_{\tilde{L}})^{-1}$. Note that ${}^\star$
is indeed well defined on $L$. Given  an arbitrary (not necessarily compact)  subset  $W \subset H$, we define $\oplam (W)$ via
\[
    \oplam (W) \; := \; \{x\in L : x^\star \in W\}.
\]
The set $W$ is then sometimes  referred to as the \textit{window}.

The preceding discussion explains all notation needed in the third point of the above theorem. While we do not give a complete proof of the  theorem here, we will include some explanation in order to give the reader some of the ideas involved. In particular, we will provide a sketch of how (iii) implies (i). To do so we first highlight two very crucial features of the construction via $\oplam$.

\begin{prop}\label{properties-oplam}
Let  $(H,\widetilde{L})$ be a cut-and-project scheme over $G$. Then,
\begin{itemize}
\item $\oplam (W)$ is relatively dense if the interior of $W$ is
non-empty.
\item  $\oplam (W)$ is uniformly discrete if the closure of $W$ is
compact.
\end{itemize}
\end{prop}

Now, whenever a Delone set $\Del$ is contained in $\oplam (W)$ with $W$ compact, then $\Del - \Del$ is contained in $\oplam (W) - \oplam (W) \subset \oplam (W - W)$. As $W - W$ is compact, we infer uniform discreteness of $\Del - \Del$. In fact, this argument can be extended to give that
any set of the form $\Del \pm \cdots \pm \Del$  (with finitely many terms) is uniformly discrete.

We can now provide a proof for (iii)$\Longrightarrow$ (i) as follows. As $\Del$ is a Delone set, there exists a compact  $K\subset G$ with $\Del + K = G$. By (iii) and the argument we just gave we have that
$$ F:= (\Del - \Del - \Del)\cap K$$
is finite. Consider now arbitrary $x,y\in \Del$. Then,
$$ x-y = z + k$$
for some $k\in K$ and $z\in \Del$ (as $\Del + K = G$). Now $k$ satisfies $k = x - y - z$ and hence belongs to $\Del - \Del - \Del$ as well. These considerations show   that $k$ belongs to the finite set $F$, finishing the proof.

\medskip

\textbf{Remark.} If the group  $G$ is compactly generated  (as is the case for $G = \R^N$)  even more is known. In this case the Delone set $\Del$ satisfies $\Del - \Del \subset \Del +  F$ for some finite $F$  if and only if $\Del - \Del $ is uniformly discrete. This was first shown by Lagarias for $G=\R^N$ \cite{Lagarias96}. As discussed  in \cite{BaakeLenzMoody07} the proof  carries over to compactly generated $G$.

\medskip

\begin{definition} (Meyer set) A Delone set $\Del \subset G$ is called Meyer if it satisfies one  of the equivalent properties of the preceding theorem. The dynamical system induced by a Meyer set is called a Meyer dynamical system.
\end{definition}

It is worth noting that any Meyer set is an FLC-Delone set. Indeed, it is a Delone set by definition. Moreover, by the  first property in Theorem \ref{theorem-characterization-Meyer}  it satisfies $\Del -  \Del \subset \Del + F$ for a finite $F$. This immediately implies that $\Del - \Del$ is locally finite and, hence, $\Del$ has  finite local complexity.  Let us also emphasize that the Meyer property is substantially stronger than the FLC property.

 \medskip

Meyer sets can be further distinguished depending on properties of $W$:

An {\em inter  model set},  associated with the  cut-and-project scheme $(H,\widetilde{L})$, is a non-empty subset  $\Del$ of $G$ of the
form
\[
x + \oplam (y + W^\circ) \subset     \Del \subset x   + \oplam (y + W),
\]
where $x\in G$, $y\in H$, and $W\subset H$ is compact with $$W=\overline{W^\circ}.$$
Such an inter model set is called   {\em regular\/} if the Haar measure of the boundary $\partial W$ of $W$ is zero.

Observe that any inter  model set is a Meyer set, i.e.,  a Delone set contained in some $\oplam (W)$ for $W$ compact.
Indeed,  by construction it is contained in some $\oplam (W)$ with $W$ compact and the Delone property follows from Proposition \ref{properties-oplam}.

\medskip

Our results below show that it is worth  providing special names for inter model sets which are repetitive.

\begin{definition}(Complete Meyer set) A set $\Del$ in $G$ is called a complete Meyer set if it is repetitive and there exists a cut  and project scheme $(H,\widetilde{L})$ and a compact $W\subset H$ with $W=\overline{W^\circ}$ such that
$x + \oplam (y + W^\circ) \subset     \Del \subset x   + \oplam (y + W)$. If furthermore the Haar measure of the boundary of $W$ is zero, the set $\Del$ is called a  regular complete Meyer set.
\end{definition}

\textbf{Remark.} Given $y\in H$, if $y+\partial W$ does not intersect $L^*$ one has
$\oplam (y + W^\circ) = \oplam (y + W)$ and so the inter model set defined by such parameters $y$ and $x\in G$ equals $x+\oplam (y + W)$ and is repetitive. It is then referred to as a model set (or a cut-and-project set). Any repetitive model set is thus a complete Meyer set. In order to carefully state our results also in the case of singular choices for $y$ (that is, $y$ for which $y+\partial W$ meets $L^*$) we use the new name `complete Meyer set'.
So strictly speaking, some complete Meyer sets are not model sets 
(for the given CPS and window)\footnote{In some works the terminology repetitive model set is used for any element in the hull of a model set with non-singular parameter $y$.
With this usage a complete Meyer set and a repetitive model set are the same thing.}. Furthermore, the name complete Meyer set suggests a process of 
completing a Meyer set, namely to a repetitive inter model set it sits in. Concrete ideas about that can be found in \cite{AujoguePhD}.
\smallskip

\textbf{Remark.} Regular complete Meyer sets have been a prime source of models for quasicrystals. The reason is that these sets have (pure) point diffraction and this is a characteristic feature of quasicrystals. Indeed,  a rigorous mathematical framework for diffraction was given    by Hof in \cite{Hof}. In this work it is also shown that regular complete Meyer sets have a lot of point diffraction.   That this spectrum  is pure was then shown later in \cite{Hof2,Schlottmann00}. These works actually prove that 
 the dynamical systems arising from regular complete Meyer sets have pure point dynamical spectrum (with continuous eigenfunctions). This is then combined with a result originally due to Dworkin \cite{Dwo}   giving that pure point dynamical spectrum implies pure point  diffraction. Recent years have seen quite some activity towards a further understanding of this result of Dworkin. In this context we mention results on a converse (i.e., that pure point diffraction implies pure point spectrum) obtained in \cite{LMS,BaakeLenz04,Gouere05,LenzStrungaru09}. The relevance of regular complete Meyer sets  in the study of quasicrystals has been underlined by recent results showing that
 the other main class of examples - those arising from primitive substitutions - is actually a subclass  if pure point diffraction is assumed (compare  Corollary \ref{subst-complete-Meyer} and subsequent discussion).

\medskip

As is clear from the discussion, any lattice  $L$ in $G$ is a Meyer set as well. In this case we can take $H $ to be the trivial group and $W = H$ and this gives that a lattice is a regular complete Meyer set. Later we will also encounter  Delone sets $\Del$ whose periods
 $$P (\Del):=\{t\in G: \Del -t = \Del\}$$
  form a lattice. Such a Delone set is  called \textit{completely periodic}  (or \textit{crystalline}). Note that the periods of a Delone set automatically form a discrete subgroup of $G$. Thus, the requirement of complete periodicity is really that the periods form a relatively dense set.  It is not hard to see that any completely periodic Delone set $\Del$ has the form
  $$\Del = L + F$$
  with a lattice $L$ (viz the periods) and a finite set $F$. This easily gives that
 any completely periodic Delone set is a Meyer set again.

\medskip

Any Meyer set over $G$ gives rise to a rotation on an compact abelian group. This will play quite a role in the subsequent analysis. By (iii) of Theorem~\ref{theorem-characterization-Meyer}, given a Meyer set there exists a cut-and-project scheme $(H,\widetilde{L})$ over $G$.
As   $\tilde{L}$ is  a discrete and co-compact subgroup of
$G\times H$, the quotient
\[
   \TT \; := \; (G\times H) / \tilde{L}
\]
is a compact abelian group and there is a natural group homomorphism
  $$ G \longrightarrow \TT,\:t \mapsto (t,0) + \tilde{L}.$$
  In this way, there is natural action of $G$ on $\TT$ and  $(\TT,G)$ is a rotation over $G$.  
%
The system $(\TT,G)$ is sometimes referred to as the canonical torus associated to the cut-and-project scheme \cite{BaakeMoodyPleasants}.

\section{Equicontinuous factors and Delone  dynamical systems}\label{Equicontinuous}
In this section  we will deal with special dynamical systems. When defining the concepts of equicontinuity and of almost periodicity we will assume  that the topology of $X$ comes from a metric $d$.  We will then refer to the corresponding dynamical systems as \textit{metrizable dynamical systems}.
This assumption of metrizability  is not necessary for the subsequent results to hold.  In fact, in order to formulate the concepts and prove the results,  it suffices to have a topology generated by a uniform structure. In particular, the results apply to Delone dynamical systems on arbitary locally compact, $\sigma$-compact abelian groups (as these systems can be topologized by a uniform structure \cite{Schlottmann00}). So, in order to simplify the presentation we will define the concepts in the metric case only, but state the results for the general case. We refer the reader to Auslander's book \cite{Auslander88} for more detailed information  and further background.

\subsection{Equicontinuous actions}
In this section we recall some of the theory of equicontinuous systems. Further aspects related to proximality will be discussed later in Section \ref{Proximality}.

Consider a minimal dynamical system $(X,G,\alpha)$ with $X$ a compact metric space and $G$ a locally compact abelian group acting by $\alpha$ on $X$.
If the action is free (that is, $\alpha_t(x)=x$ for some $x\in X$ implies $t$ is the identity of $G$) then $X$ can be seen as a compactification of $G$: it is the completion of one orbit and this orbit is a copy of $G$. One might ask when is $X$ a group compactification, that is, when does $X$ carry a group structure such that the orbit is a subgroup isomorphic to $G$, or, in other words, when is $(X,G)$ a rotation on a compact abelian group?
An answer to this question can be given in terms of equicontinuity.
\begin{definition} (Equicontinuous system)
The metrizable dynamical system $(X,G,\alpha)$ is called equicontinuous
if the family of homeomorphisms $\{\alpha_t \}_{t\in G}$ is equicontinuous, i.e., if for all $\varepsilon >0$ there exists a $\delta>0$ such that
$$ d(\alpha_t(x),\alpha_t(y))<\epsilon$$
for all $t\in G$ and all $x,y\in X$ with $d(x,y)<\delta$.
\end{definition}

\textbf{Remark.} The definition of equicontinuity may seem to depend on the particular choice of the  metric. However, by compactness of the underlying space $X$ is turns out to hold for one metric if and only if it holds for any metric (which induces  the topology).

\medskip

An equicontinuous  system admits an invariant metric which induces the same topology. Indeed one can just take
$$\overline{d}(x,y) := \sup_{t\in G} d(t \cdot x,t \cdot y).$$
Likewise, any compact metrizable abelian group $\TT$ admits a left invariant metric: Whenever $d$ is a metric  then
$\overline{d}(x,y) := \sup_{t\in \TT} d(x-t,y-t)$ is a metric on $\TT$ which is invariant. This similarity is not a coincidence: According to the following well-known theorem (see, for example, \cite{Kurka}) equicontinuous minimal systems are rotations on groups.


\begin{theorem} (Equicontinuous systems as
 rotations on compact groups) \label{cor-equi} The minimal dynamical system $(X,G)$ is equicontinuous if and only if it
is conjugate to a minimal rotation on a compact abelian group.
\end{theorem}
  \textit{Sketch of proof:} A rotation on a compact abelian group  is obviously  equicontinuous. Conversely,
if $(X,G)$ is equicontinuous, given any point $x_0\in X$ the operation $t_1\cdot x_0 + t_2\cdot x_0 := (t_1 + t_2)\cdot x_0$ extends to an addition in $X$ so that
$X$ becomes a group with $x_0$ as neutral element.   \hfill \qed

\medskip

\textbf{Remark.} An equicontinuous dynamical system need not be minimal but a transitive equicontinuous dynamical system  (i.e., one containing a dense orbit) is always minimal. So in the context of Delone (and tiling - see Section \ref{Meyer substitutions}) dynamical  systems, which are by definition the closure of one orbit,  equicontinuous systems are always minimal.

\medskip

Equicontinuity is strongly related to almost periodicity. In order to explain this further we will need some notation.
Let $(X,G)$ be a metrizable  dynamical system and $d$ the  metric on $X$. Then, the $\epsilon$-ball around $x\in X$ is denoted by $B_\epsilon (x)$. The elements of
 $$\Rr(x,\epsilon):=\{t\in G: t\cdot x\in B_\epsilon(x)\}$$
are called \textit{return vectors} to $B_\epsilon (x)$.
 Now, $(X,G)$ is called
\textit{uniformly almost periodic} if, for any $\epsilon>0$ the joint set of return vectors to $\epsilon$-balls, given by
$$A = \bigcap_{x\in X} \Rr(x,\epsilon)$$
is relatively dense (i.e., there exists a compact $K$ with $A + K = G$).

\begin{theorem}[\cite{Auslander88}] (Equicontinuity via almost periodicity)\label{thm-Auslander-equi}  The minimal dynamical system
$(X,G)$ is equicontinuous if and only if it is uniformly almost periodic.
\end{theorem}

For Delone dynamical systems we can even be more specific.  Recall that a continuous bounded function $f$ on $G$ is called {\em Bohr almost periodic} if for any $\varepsilon >0$ the set
$$\{t\in G : \|f - f(\cdot -t)\|_\infty \leq \varepsilon\}$$
of its {\em almost $\varepsilon$-periods} is relatively dense.

\begin{theorem} (Characterization equicontinuous Delone systems)  Let $\Del$ be a Delone set in the locally compact, $\sigma$-compact abelian group  $G$. Then, the following assertions are equivalent:

\begin{itemize}
\item[(i)] The function
 $$f_{\Del,\varphi} : G\longrightarrow \C, \:\; f_{\Del, \varphi} (t) = \sum_{x\in \Del} \varphi (t - x)$$
 is Bohr-almost periodic for any continuous and compactly supported function $\varphi$ on $G$.
\item[(ii)] The hull $\Omega_\Del$ is  a compact abelian group   with  neutral element $\Del$ and group addition  satisfying $(\Del-t) + (\Del-s) = (\Del-s -t)$ for all   $t,s\in G$.
\item[(iii)]  The dynamical system $(\Omega_\Del,G)$ is equicontinuous.
\end{itemize}
\end{theorem}
The theorem (and a proof) can be found in \cite{KellendonkLenz13}. Of course, the equivalence between (ii) and (iii) follows from the above Theorem \ref{cor-equi}.  The equivalence between (i) and (ii) is close in spirit to Theorem \ref{thm-Auslander-equi}. However, the  proof given in \cite{KellendonkLenz13} is based on \cite{LenzRichard07}.

\bigskip

It is possible to describe (up to conjugacy) all equicontinuous systems over $G$ via   subgroups of $\Gdual$. The reason is basically that an equicontinuous system is a rotation on a compact group due to Theorem \ref{cor-equi}. This  compact group in turn is determined by its dual group, which is just a subgroup of $\Gdual$.  Moreover, the  elements of this dual group turn out to be just the continuous  eigenvalues of the system.  This highlights  the role of the continuous eigenvalues.

As it  is both instructive in itself  and also enlightening for the material presented  in the next section we now give a more detailed discussion of these connections.  We start with a general construction and then state the result describing the equicontinuous minimal systems over $G$.

\smallskip

Let $\mathcal{E}$ be a  subgroup of $\Gdual$.
We equip $\mathcal{E}$ with the discrete topology and denote the  dual group of $\mathcal{E}$ by $\TT_\mathcal{E}$, i.e.,
$$\TT_\mathcal{E} := \widehat{\mathcal{E}}.$$
Then, $\TT_\mathcal{E}$ is a compact abelian group. The inclusion  
$\mathcal{E} \hookrightarrow \Gdual$ gives rise (by Pontryagin duality) to a group homomorphism $j : G \longrightarrow \TT_\mathcal{E}$ with dense range. In this way there is a natural action of $G$ on $\TT_\mathcal{E}$ via
$$ G\times \TT_\mathcal{E}\longrightarrow \TT_\mathcal{E}, \; (t,x)\mapsto t\cdot x := j(t) x,$$
where on the right hand side the elements $j(t)$ and $x$ of $\TT_\mathcal{E}$ are just multiplied via the group multiplication of $\TT_\mathcal{E}$.  This action can be explicitly calculated as
$$ (t\cdot x) (\chi) = (\chi,t) (x,\chi)$$
for $\chi \in \mathcal{E}$.  As $j$ has dense range, this group action is minimal and hence (as $\TT_\mathcal{E}$ is compact), it is uniquely ergodic with the (normalized) Haar measure on $\TT_\mathcal{E}$ as the invariant measure.  Thus, $(\TT_\mathcal{E},G)$ is a minimal uniquely ergodic  rotation on a compact group. 
By  Theorem \ref{cor-equi}, such a system is   equicontinuous. Furthermore, countability of $\mathcal{E}$    is equivalent to metrizability  of the dual group $\TT_{\mathcal{E}}$ by standard harmonic analysis.

\begin{theorem} \label{Description-equicontinuous-systems-over-G} Let $G$ be a locally compact abelian group.

(a) Let $(X,G)$ be an  equicontinuous minimal dynamical system. Then,
  $(X,G)$ is conjugate to $(\TT_{\mathcal{E}_{ {top}} (X,G)},G)$.

(b) Whenever $\mathcal{E}$ is a subgroup of $\Gdual$ then $(\TT_{\mathcal{E}},G)$ is the  unique (up to conjugacy) equicontinuous minimal dynamical system whose set of continuous eigenvalues is given by $\mathcal{E}$. \end{theorem}
\begin{proof} A sequence of claims establishes the statements of the theorem.

\smallskip

\textit{Claim 1}.
 Let $(\TT,G)$ be  a minimal rotation on the  compact group $\TT$  with  action of $G$ induced by the homomorphism $j : G\longrightarrow \TT$. Let $\iota : \widehat{\TT} \longrightarrow \Gdual $ be the dual of $j$ (i.e., $\iota(\chi) (t) =   (\chi , j( t))$ for $\chi \in \widehat{\TT}$ and $t\in G$). Then, $\iota$ is injective and the  set $\mathcal{E}$ of  continuous  eigenvalues of $\TT$ is just the image of  $\widehat{\TT}$  under $\iota$.
 In this way, the system $(\TT,G)$ is completely determined by the set of its continuous eigenvalues.

\textit{Proof.}
 The action of $G$ on $\TT$ comes from
 a homomorphism $j: G\longrightarrow \TT$ with  dense range. Thus, its dual map $\iota : \widehat{\TT}\longrightarrow \Gdual$ is injective. We therefore have to show that the continuous eigenvalues of $(\TT,G)$ are just given by the $\iota (\widehat{\TT})$. This in turn follows easily from the definitions:
Any $\chi$ in the dual group of $\TT$  gives rise to the continuous  $f_\chi : \TT  \longrightarrow \C, f_\chi (x) = (\chi,x)$ which takes the value $1$ at the neutral element of $\TT$. This $f_\chi$ is an eigenfunction to $\iota (\chi)$ as
$$f_\chi (t\cdot x) =  (\chi, j(t) x  )=  (\chi,x)  (\chi,j(t)) =  (\iota  (\chi),t) f_\chi (x).$$  Conversely,  whenever  $\lambda $ is a continuous eigenvalue of $(\TT,G)$ with associated eigenfunction $f_\lambda$ we can assume without loss of generality that $f_\lambda $ takes the value $1$ on the neutral element of $\TT$. Then the eigenvalue equation gives that $f_\lambda$ is actually an element of the dual group of $\TT$.

\medskip

\textit{Claim 2.} Consider the map $\iota $ from the previous claim as a bijective group homomorphism from  the discrete group  $\widehat{\TT}$ onto the discrete group $\mathcal{E}$.  Let $\kappa$ be its dual  mapping    $ \TT_\mathcal{E} = \widehat{\mathcal{E}}$ to $\TT$. Then $\kappa$ establishes a topological conjugacy between  $\TT_\mathcal{E}$ and $\TT$.

\textit{Proof.} This follows directly from unwinding the definitions.

\medskip

\textit{Claim 3.} Let $\mathcal{E}$ be a  subgroup of $\Gdual$ and $(\TT_\mathcal{E},G)$ be the associated equicontinous dynamical system. Then the  set $\mathcal{E}_{top} (\TT_\mathcal{E},G)$ of continuous eigenvalues of $(\TT_\mathcal{E},G)$ can naturally be identified with  $\mathcal{E}$.

\textit{Proof.}  The previous claim  shows that the set of continuous eigenvalues is just given by the dual of the group $\TT_\mathcal{E}$. By construction this dual is just $\mathcal{E}$.

\smallskip

The statements of the theorem follow directly from the above claims.
\end{proof}

\smallskip

\textbf{Remark.} Part (b) of the previous theorem concerns the construction of equicontinuous systems with a given group of continuous eigenvalues. In the context of the present paper it seems appropriate to point out that (for subgroups of the Euclidean space)  such a system can even be constructed as the torus of a cut-and-project scheme \cite{Robinson07}.

\subsection{Maximal equicontinuous factor}
In this section we consider   a topological minimal  dynamical system  $(X,G)$.  There exists a largest (in a natural sense) equicontinuous factor of this system. It is known as {\em maximal equicontinuous factor}. This factor can be obtained in various ways  including
 \begin{itemize}
\item via  the dual of the  topological eigenvalues,
\item via  a quotient construction,
\item via the  Gelfand spectrum of continuous eigenfunctions.
\end{itemize}
All of this is certainly well known. In fact, substantial parts  can be found in the book \cite{Auslander88} for example. Other parts of the theory, while still quite elemenary, seem to be scattered over the literature.  In  particular, in the context of our interests, corresponding constructions are discussed in \cite{BaakeLenzMoody07,BargeKellendonk12}.
 For the convenience of the reader we give a rather detailed discussion here.

\bigskip

Recall that $\mathcal{E}_{top}(X,G)$ carries the discrete topology.  Thus we are exactly in the situation discussed at the end of the previous section. In particular, there is a
 group homomorphism $j : G \longrightarrow \TT_{\mathcal{E}_{top} (X,G)}$ with dense range. This group homomorphism induces an action of $G$ on $\TT_{\mathcal{E}_{top} (X,G)}$ making it into a rotation on a compact group. In this way we obtain a minimal uniquely ergodic dynamical system $(\TT_{\mathcal{E}_{top} (X,G)},G)$ out of our data.  Explicitly, the  action of $G$ on $\TT_{\mathcal{E}_{top} (X,G)}$ is given as $$ (t\cdot x) (\chi) = (\chi,t) (x,\chi)$$
for $\chi\in \mathcal{E}_{top}(X,G)$.
We present two remarkable properties of this dynamical system in the next two lemmas.

\begin{lemma} (Description of equicontinuous factors) \label{description-equicontinuous-factors} Let $(X,G)$ be a minimal dynamical system and $(Y,G)$ an equicontinuous factor. Then $\mathcal{E}_{top}(Y,G)$ is a subgroup of $\mathcal{E}_{top}(X,G)$ and  $(Y,G)$ is a factor of  $(\TT_{\mathcal{E}_{top}(X,G)},G)$.
\end{lemma}
\begin{proof} Let $\pi : X \longrightarrow Y$ be the factor map. Then any continuous eigenfunction $f$  on $Y$ (to the eigenvalue $\chi$) gives rise to the eigenfunction $f\circ \pi$ on $X$ (to the eigenvalue $\chi$). This shows the first part of the statement. Obviously, the embedding
$$ \mathcal{E}_{top}(Y,G) \longrightarrow \mathcal{E}_{top}(X,G)$$
is a group homomorphism.  Dualising, we obtain a group homomorphism
$$\TT_{\mathcal{E}_{top}(X,G)} = \widehat{\mathcal{E}_{top}(X,G)} \longrightarrow \widehat{\mathcal{E}_{top}(Y,G)} = \TT_{\mathcal{E}_{top}(Y,G)}.$$
This group homomorphism can easily be seen to be  a $G$-map. Hence, by compactness and minimality of the groups in question, it is onto and hence a factor map. The desired statement now follows as $\TT_{\mathcal{E}_{top} (Y,G)}$ is  conjugate to $(Y,G)$ by Theorem \ref{Description-equicontinuous-systems-over-G}.
\end{proof}

\textbf{Remark.} Note that the factor map from   $(\TT_{\mathcal{E}_{top}(X,G)},G)$ to $(Y,G)$ can  be chosen to be a group homomorphism. This is clear from the proof. In fact, it is  a general phenomenon: As is  easily shown, if a rotation on a compact group is a factor of another rotation on a compact group (mapping the neutral element to the neutral element), then the factor map is a group homomorphism (see also the proof of (b) of Theorem 3.1 in \cite{LenzRichard07} for this type of reasoning).

\medskip

\begin{lemma} Let $(X,G)$ be a minimal dynamical system. Then $(\TT_{\mathcal{E}_{top} (X,G)},G)$ is a factor of $(X,G)$.
\end{lemma}
\begin{proof} Fix an arbitrary point $x_0\in X$ and choose for any $\chi \in \mathcal{E}_{top} (X,G)$ a continuous eigenfunction $f_\chi$ with $f_\chi (x_0) =1$. Using the eigenfunction equation along the orbit of $x_0$ and minimality we have:
$$f_{\chi \eta}  = f_\chi f_\eta,\:\; \overline{f_\chi} = f_{-\chi}.$$

Any $x\in X$ then gives rise to the map
$$\widehat{x} : \mathcal{E}_{top} (X,G) \longrightarrow S^1, \widehat{x} (\chi) := f_\chi (x).$$
By construction and the choice of the $f_\chi$ the map $\widehat{x}$ is a character on $\mathcal{E}_{top} (X,G)$, i.e., an element of $\TT_{\mathcal{E}_{top} (X,G)}$. It is not hard to see that the map
$$X\longrightarrow \TT_{\mathcal{E}_{top} (X,G)}, x\mapsto \widehat{x},$$
is a $G$-map. By minimality and compactness, it is then a factor map.
\end{proof}

The two previous lemmas establish the following theorem.

\begin{theorem}\label{characterization-maximal-equicontinuous-factor} Let $(X,G)$ be minimal. Then there exists a unique (up to conjugacy) factor $(X_{\max},G)$  of $(X,G)$  satisfying   the following two properties:
\begin{itemize}
\item The factor $(X_{\max},G)$ is equicontinuous.
\item Whenever $(Y,G)$ is an equicontinuous factor of $(X,G)$  then $(Y,G)$ is a factor of $(X_{\max},G)$.
\end{itemize}
\end{theorem}
\begin{proof}   Existence follows directly from the previous two lemmas. Uniqueness can be shown as follows: Let $(Y_1,G)$ and $(Y_2,G)$ be two factors with the above properties. Then both $Y_1$ and $Y_2$ are compact groups and  there exist factor maps
 $\pi_{1,2} : Y_1 \longrightarrow Y_2$, and  $\pi_{2,1} : Y_2\longrightarrow Y_1$.  Without loss of generality we can assume that $\pi_{1,2}$  maps the neutral element $e_1$ of $Y_1$ onto the neutral element, $e_2$, of $Y_2$ and $\pi_{2,1}$ maps $e_2$ to $e_1$ (otherwise we could just compose the maps with appropriate rotations of the groups). Then
 $$ \pi_{1,2} \circ \pi_{2,1} (e_1) = e_1\;\:\mbox{and}\:\; \pi_{2,1} \circ \pi_{1,2} (e_2) = e_2.$$
 As both $\pi_{1,2}$ and $\pi_{2,1}$ are $G$-maps, this shows that $\pi_{1,2} \circ \pi_{2,1}$ agrees with the identity  on $Y_1$ on the whole orbit of $e_1$ and $\pi_{2,1} \circ \pi_{1,2}$ agrees with the identity on $Y_2$ on the whole orbit of $e_2$. By continuity of the maps and denseness of the orbits we infer that the maps are inverse to each other. This shows that indeed $(Y_1,G)$ and $(Y_2,G)$ are conjugate.
 \end{proof}

\begin{definition} The factor $(X_{\max}, G)$ is called the maximal equicontinuous factor of $(X,G)$. The corresponding factor map will be denoted by $\pi_{\max}$.
\end{definition}

We now present two additional ways to view the maximal equicontinuous factor. 

\medskip

\textbf{A construction via quotients.} Let the equivalence relation $\sim$ on $X$ be defined by $x\sim y$ if and only if $f (x) = f(y)$ for every continuous eigenfunction $f$  and let
$$\pi : X \longrightarrow X / \sim =: X_\sim$$
be the canonical projection. Let $X_\sim$ have the quotient topology so that a map $g$ on $X_\sim$ is continuous if and only if $g\circ \pi$ is continuous.
 It is not hard to see that the action of $G$ on $X$ induces an action of $G$ on $X_\sim$  by the (well defined!)  map
$$G\times X_\sim \longrightarrow X_\sim, (t,\pi (x)) \mapsto (\pi (t\cdot x)).$$
Then $\pi$ is a $G$-map and hence a factor map. Note that the preceding considerations show that
whenever $f_\chi $ is a continuous eigenfunction to the eigenvalue $\chi$  there exists a unique continuous eigenfunction $g_\chi$ on $X_\sim$ with $f_\chi = g_\chi \circ \pi$.

If we are given additionally an invariant probability measure $m$ on $X$, this measure is transferred to  a $G$-invariant measure  $m_\sim :=\pi (m)$ on $X_\sim$. In this way we have constructed a  dynamical system $(X_\sim, G)$ together with an invariant measure $m_\sim$.

\medskip

\textbf{A construction via the Gelfand transform.}  Let $\mathcal{A}$ be the closed (w.r.t to $\|\cdot\|_\infty$) subalgebra of $C(X)$ generated by the continuous eigenfunctions. Then $\mathcal{A}$ is a commutative $C^\ast$-algebra and there exists therefore a compact space $X_\mathcal{A}$ and a continuous isomorphism of algebras (Gelfand transform)
$$\Gamma : \mathcal{A}\longrightarrow C (X_\mathcal{A}).$$
The space $X_\mathcal{A}$ is in fact nothing but the set of all multiplicative linear non-vanishing functionals on $\mathcal{A}$ and the map $\Gamma $ is then given by
$$\Gamma (f) (\phi) = \phi (f)$$
for $f\in \mathcal{A}$ and $\phi \in X_\mathcal{A}$.
The action of $G$ on $X$ induces an action of $G$ on $\mathcal{A}$ and this in turn induces an action of $G$ on $X_\mathcal{A}$. By construction, $\Gamma$ is then  a $G$-map with respect to these actions, i.e.,
$$\Gamma ( f(t\cdot)) = (\Gamma f) (t\cdot).$$

 Assume now that we are additionally given an invariant probability measure $m$ on $X$. Then $m$  can be seen as a linear positive functional $m : \mathcal{A}\longrightarrow \CC, f\mapsto m(f)$. Thus, via $\Gamma$, it induces a linear positive functional $m_\mathcal{A}$ on $C(X_\mathcal{A})$ and hence a  $m_{\mathcal{A}}$ is a measure on $X_{\mathcal{A}}$. It is not hard to see that the map $\Gamma : \mathcal{A} \longrightarrow C (X_\mathcal{A})$ extends to a unitary $G$-map
$$U : L_{pp,top}^2 (X, m)\longrightarrow  L^2 (X_\mathcal{A}, m_{\mathcal{A}}),$$
where $L_{pp,top}^2 (X,m)$ is the subspace of $L^2 (X,m)$ generated by the continuous eigenfunctions. (The subscript $pp$ in the notation refers to `pure point'.) As is easily seen, the only $G$-invariant functions on $L^2 (X_\mathcal{A}, m_\mathcal{A})$ are constant. Thus $m_\mathcal{A}$ is an ergodic measure on $(X_\mathcal{A}, G)$ and we have   expressed $L^2_{pp,top}$ as the $L^2$-space of a dynamical system.

\medskip

We now discuss how all three constructions give the same dynamical system (up to conjugacy).

\begin{theorem} \label{conjugate} Let $(X,G)$ be a minimal dynamical system. Then  the  dynamical systems $(X_\mathcal{A},G)$, $(\TT_{\mathcal{E}_{top}  (X,G) },G)$ and $(X_\sim, G)$ are canonically conjugate. In particular, they are all uniquely ergodic and  minimal and have pure point spectrum.
\end{theorem}
\begin{proof} Chose for any $\chi \in \mathcal{E}_{top} (X,G)$ a continuous eigenfunction $f_\chi$. Let   $g_\chi$ on $X_\sim$ be the unique function with $f_\chi = g_\chi \circ \pi$. Fix an $x_0 \in X$ arbitrarily.
 We can  assume without loss of generality that $f_\chi (x_0) = 1$ for all $\chi\in \mathcal{E}_{top} (X,G)$. Write  $\TT$ for $\TT_{\mathcal{E}_{top}  (X,G)}$.

We first show the statement on canonical conjugacy. To do so we  provide explicit maps: Define
$$ J : X_\sim \longrightarrow \TT \;\:\mbox{via}\;\:  J(\pi (x) ) (\chi ) = g_\chi (\pi (x)) = f_\chi (x).$$
Then $J$ indeed maps $X_\sim $ into $\TT$ (as we had normalized our $f_\chi$ with $f_\chi (x_0) =1$). Obviously, $J$ is continuous and injective. As it is a $G$-map and the action of $G$ on $\TT$ is minimal, the map $J$  has dense range. As $\TT$ is compact, $J$ is then a homeomorphism.

\medskip

We now turn to proving that $X_\sim $ and $X_\mathcal{A}$ are homeomorphic.
Consider
$$\Pi : C (X_\sim) \longrightarrow  C(X), g\mapsto g\circ \pi.$$
Then  $B:= \Pi (C (X_\sim))$ is a closed subalgebra of $C(X)$. As $C(X_\sim)$ is generated by the $g_\chi$, $\chi \in\mathcal{E}_{top} (X,G)$, the algebra $B$ is generated by $f_\chi = g_\chi \circ \pi$, $\chi \in\mathcal{E}_{top} (X,G)$. Therefore $B = \mathcal{A}$ and thus $\Pi$ gives an isomorphism between $C (X_\sim)$ and $\mathcal{A}$. Dualising $\Pi$, we obtain a homeomorphism between $X_\sim$ and $X_\mathcal{A}$.  It is easy to check that all maps involved are $G$-maps.

\smallskip

The last statement of the theorem is clear since $(\TT,G)$ is uniquely ergodic, minimal, and has pure point spectrum.
\end{proof}

 The preceding considerations are \textbf{summarized} as follows:
 We have given three different constructions of a certain topological factor of $(X,G)$. This factor is given by a rotation, i.e., an action of the group $G$ on a compact group via a group  homomorphism  with dense range from $G$ to the compact group. The $L^2$-space of this factor corresponds to the  part of the $L^2$-space of the original  dynamical system coming from continuous eigenfunctions.

\bigskip

In general it is not easy to decide whether a given equicontinuous factor is  the maximal equicontinuous factor. However, there is one sufficient condition which is of considerable relevance.

\begin{lemma}\label{lemma-1-1-almost-automorphic} Let $(X,G)$ be a minimal  dynamical system and $(Y,G)$ an  equicontinuous dynamical system. If $(Y,G)$ is a factor of $(X,G)$ with factor map $\pi$  and there exists  $y\in Y$ such that $\pi^{-1} (y)$ consists of only one point, then $(Y,G)$ is the maximal equicontinuous factor.
\end{lemma}
\begin{proof} Let $(X_{\max},G)$ be the maximal equicontinuous factor with corresponding factor map $\pi_{\max}$. By Theorem \ref{characterization-maximal-equicontinuous-factor},  the dynamical system $(Y,G)$ is then a factor of $(X_{\max},G)$. Denote the corresponding factor mapy by $\pi_Y$.
Without loss of generality we can then assume that $\pi_Y$ maps the neutral element of $X_{\max}$ to the neutral element of $Y$ and that
$$\pi = \pi_Y \circ \pi_{\max}$$
(otherwise we can just compose $\pi_Y$ and $\pi$ with suitable rotations).
 We will show that $\pi_Y$ is a homeomorphism. As $\pi_Y$ is a factor map, it is onto and continuous. It therefore suffices to show that it is one-to-one. So, let $p,q\in Y$ be given with $\pi_Y (p) = \pi_Y (q)$.  As discussed in a remark above, the map $\pi_Y$ is a group homomorphism. Thus,  we obtain
$$\pi_Y (g p) = \pi_Y (g) \pi_Y (p) = \pi_Y (g) \pi_Y (q) = \pi_Y (g q)$$
for all $g\in Y$. As $\pi_Y$ is onto, we can now chose $g\in Y$ with $\pi_Y (g p) = y = \pi_Y (g q)$.
As, by the assumption of the lemma, $$\pi^{-1} (y) = \pi_{\max}^{-1} ( \pi_Y^{-1} (y))$$ consists of only one point, we obtain from the last equality
that $g p = g q$ and, hence, $p = q$. This is the desired injectivity.
\end{proof}

The class of dynamical systems appearing in the previous lemma is rather important (for us, in describing regularity properties of Meyer sets - see Theorem \ref{Meyer-equal-to-mef-one-point} below) and has a name of its own.

\begin{definition}[Almost-automorphic system] Let $(X,G)$ be a minimal dynamical system and $(X_{\max}, G)$ its maximal equicontinuous factor. If there exists a $y\in X_{\max} $ such that $\pi_{\max}^{-1} (y) $ has only one element then $(X,G)$ is called almost-automorphic.
\end{definition}

\medskip

We finish this section with a discussion of  local freeness in our context.
This will  be relevant  in the discussion in Section \ref{Proximality-Delone}.
The equivalence of (i) and (ii) in the following result is  due to Barge and Kellendonk \cite{BargeKellendonk12}.  We include a complete proof (as \cite{BargeKellendonk12} only contains a proof of one direction).

\smallskip

Whenever the group $G$ acts on $X$ an element  $t$ of $ G$ is said to act \textit{freely} if $t \cdot x\neq x$ for all $x\in X$.  The action is  called \textit{locally free} if there exists a neighborhood $U$ of $e\in G$ such that any $t\in U\setminus \{e\}$ acts freely. If $U$ can be chosen as $G$ the action is called \textit{free}.

\begin{lemma}\label{lem-free} Let $\mathcal{E}$ be a subgroup of $\Gdual$ and consider the associated dynamical system $(\TT_{\mathcal{E}}, G)$ with action  given by
$$(t\cdot x)(\chi) = (\chi,t) (x,\chi)$$
 for $x\in \TT_{\mathcal{E}}$ and $\chi \in \mathcal{E}$. Then the following assertions are equivalent:
 \begin{itemize}
 \item[(i)] The action is locally free.
 \item[(ii)]  The quotient $\Gdual  / \overline{\mathcal{E}}$ is compact, where $\overline{\mathcal{E}}$ is the closure of $\mathcal{E}$ in $\Gdual$.
   \item[(iii)]   The stabiliser  $\{ t\in G: t\cdot x = x\:\mbox{ for (some) all  x} \}$ of the action is a discrete subgroup of $G$.
    \end{itemize}
In particular, the action is free if and only if $\mathcal{E}$ is dense in $\Gdual$.
\end{lemma}
\begin{proof} The equivalence between (i) and (iii) is clear. It remains to show the equivalence between (i) and (ii).
 By definition $t\in G$ acts freely if and only if $t\cdot x \neq x$ for all $x\in\TT_{\mathcal{E}}$, which is the case if and only if  $(\chi,t) \neq 1$ for at least one $\chi\in\mathcal{E}$. By continuity, the latter can be rephrased as $(\chi,t) \neq 1$ for at least one $\chi\in \overline{\mathcal{E}}$. Consider now  the exact sequence of abelian groups
$$ e \to \widehat{\Gdual / \overline{\mathcal{E}} } \to G \stackrel{q}{\to} \hat{\overline{\mathcal{E}}}\to e$$
which is the dual to the exact sequence $e\to \overline{\mathcal{E}}\to \Gdual \to \Gdual / \overline{ \mathcal{E}}\to e$.
Let $U$ be an open neighborhood of $e\in G$
and $e\neq t\in U$. Set $\eta =  q (t) \in \widehat{ \overline{ \mathcal{E} } }$. Then, $(\eta,\chi) = (\chi,t)$ because $q$ is dual to the inclusion $\mathcal{E} \hookrightarrow \Gdual$. Thus, there exists a $\chi$ with $(\chi, t) \neq 1$ if and only if $\eta$ does not belong to $\ker q$.
   Hence $t\cdot x \neq x$ for all $x\in\TT_{\mathcal{E}}$ if and only if  $t\notin \ker q $. Thus $G$ acts locally freely if and
 only if there exists an open $e\in U\subset G$ such that  $U\cap \ker q = \{e\}$.  By exactness of the sequence above, this the case if and only if  $ \widehat{\Gdual/ \overline{\mathcal{E}}}$ is discrete and hence if and only if  $\hat G/\overline{ \mathcal{E}}$ is compact.

  Furthermore, $G$ acts freely if and only if  the above is true for $U=G$, which is equivalent to $\mathcal{E}$ being dense in $\Gdual$.
\end{proof}

\subsection{Delone dynamical systems via their maximal equicontinuous factor} \label{Hierarchy}
In this section we will study dynamical systems  arising as the hull of  FLC-Delone sets.  The basic aim is to characterize features of the Delone set in question by how close its dynamical system is to its maximal equicontinuous factor.  A rough description of our results is that the more ordered the set is, the closer its hull is to its maximal equicontinuous factor. More precise statements will be given below. We will distinguish two situations. In one situation we are given a Meyer set and characterize it by features of its hull. In the other situation we are given the hull of an FLC-Delone set.

\medskip

In order to set the perspective on the results in the next two subsections we
briefly recall a `hierarchy of order' we have encountered within the FLC-Delone sets (see Section \ref{Background-Meyer}). Let  $\Del$ be a Delone set. Then $\Del$ has finite local complexity if and only if $\Del - \Del$ is locally finite.  Consequently, $\Del$  is Meyer if and only if $\Del -\Del\subset \Del + F$ for some finite set $F\subset G$ or, equivalently, if and only if $\Del$ is a subset of  $x + \oplam (W)$ with compact  $W\subset H$ and $x\in G$   for some cut-and-project-scheme $(\widetilde{L},H)$ over $G$. Now the following classes of  Meyer sets $\Del$ can be distinguished, each of them defined by a stronger requirement than the previous one:

\begin{itemize}

\item $\Del$ is a complete Meyer set if it is repetitive with   $x + \oplam (W^\circ) \subset \Del \subset x + \oplam (W)$ for some $x\in G$ and  some cut-and-project scheme $(\widetilde{L},H)$ over $G$ and a compact $W\subset H $ with $W = \overline{W^\circ}$.

\item $\Del$ is a regular complete Meyer  set if it is repetitive with   $x + \oplam (W^\circ) \subset \Del \subset x + \oplam (W)$  for some $x\in G$ and some cut-and-project scheme $(\widetilde{L},H)$ over $G$ and a compact $W\subset H$  with  $W = \overline{W^\circ}$ and  boundary  $\partial W$ of Haar measure zero.

\item $\Del$ is completely periodic  if the set $\{ t\in  G: t + \Del = \Del\}$ is a lattice or,  equivalently, if  $ \Del = L + F$ for a lattice $L$ and a finite set $F$.

\end{itemize}

\subsubsection{Regularity of  Meyer sets via dynamical systems}
In this section we study the hull $(\Omega_\Del,G)$ of a repetitive Meyer set $\Del$. We investigate how  the hierarchy of Meyer sets discussed  above is reflected in injectivity properties of the factor map between this hull and  its maximal equicontinuous factor.

\bigskip

\begin{theorem}\label{Meyer-equal-to-mef} Let $G$ be a locally compact abelian group and $\Del$ a repetitive Meyer set in $G$.  The following assertions are equivalent:

\begin{itemize}
\item[(i)] $\Del$ is completely periodic.

\item[(ii)] The dynamical system $(\Omega_\Del,G)$ of $\Del$ is conjugate to its maximal equicontinuous factor (i.e., each point  in the maximal equicontinuous factor  has exactly one inverse image point under the factor map).
\end{itemize}
\end{theorem}
 Note that (ii) actually says that $(\Omega_\Del,G)$ is just a rotation on a compact group. Using the material of the last section (on characterizing the maximal equicontinuous factor via a quotient construction) this can be seen to be equivalent to  the continuous eigenfunctions separating the points.  With this formulation (instead of (ii)) the result is shown in \cite{BaakeLenzMoody07}.  Of course, the implication (i)$\Longrightarrow$ (ii) is clear and it is the other implication where all the work lies. We will comment a bit on the method of proof after stating the next result.

 \bigskip

\begin{theorem}\label{Meyer-equal-to-mef-almost-everywhere}

 Let $G$ be a locally compact abelian group and let $\Del$ be a repetitive  Meyer set in $G$.  The following are equivalent:

\begin{itemize}
\item[{\rm (i)}]  $\Del$ is  a regular  complete Meyer set.

\item[{\rm (ii)}] The dynamical system $(\Omega_\Del,G)$ of $\Del$ is an almost-$1$-to-$1$ extension of its maximal equicontinuous factor (i.e., the set of points in the maximal equicontinuous factor with exactly  one inverse image point  under the factor map has full measure).
\end{itemize}
\end{theorem}
  This is a reformulation of the main result of  \cite{BaakeLenzMoody07} in terms of the maximal equicontinuous factor.   The formulation of (ii) given there  is somewhat different and says that   the continuous eigenfunctions separate almost all points and the system is uniquely ergodic and minimal. Now, by our discussion on how to obtain the maximal equicontinuous factor (via a quotient construction) and Lemma \ref{lemma-1-1-almost-automorphic},  this is just (ii).  
  
  The direction (i)$\Longrightarrow$ (ii) was first shown in the special case of the Penrose system by Robinson in \cite{Robinson96}. General complete Meyer sets were then treated by Schlottmann in \cite{Schlottmann00}.
  The main work of  \cite{BaakeLenzMoody07} is to show the implication (ii)$\Longrightarrow$ (i). There, the  \textit{diffraction} of the Meyer set plays a key role.  Condition (ii)  implies that  diffraction is a pure point measure and this can be used to introduce    new topologies on the Delone sets. Taking suitable completions of  the hull of $\Del$ in these topologies one then obtains the ingredients of a  cut-and-project scheme  via a method of \cite{BaakeMoody04}. The main work of \cite{BaakeLenzMoody07} is then to prove regularity features of the window.  This regularity is shown by an analysis of rotations on compact groups.
A crucial role in these  considerations is played by continuity of the eigenfunctions. This continuity is related to uniform existence of certain ergodic averages \cite{Robinson94,Lenz09}. As such it has also played a major role in the investigation of diffraction and the so-called Bombieri/Taylor conjecture. We refer the reader to \cite{Lenz09} for further discussion and background.

\bigskip

 \begin{theorem}\label{Meyer-equal-to-mef-one-point} Let $G=\R^N$  and let $\Del$ be a repetitive  Meyer set in $G$.  The following are equivalent:

\begin{itemize}
\item[{\rm (i)}] $\Del$ is a complete Meyer set.

\item[{\rm (ii)}] The dynamical system $(\Omega_\Del,G)$ of $\Del$ is an almost-automorphic system (i.e.,  the set of points in the maximal equicontinuous factor with exactly  one inverse image point under the factor map is non-empty).
\end{itemize}

\end{theorem}

This is one version of a   main result of the Ph.D. thesis of J.B. Aujogue \cite{AujoguePhD}. The implication (i)$\Longrightarrow$ (ii) is somewhat folklore. It is mentioned in the introduction of \cite{Robinson07} and  can rather directly be derived from  Lemma \ref{lemma-1-1-almost-automorphic} and some Baire type arguments (see   \cite{Schlottmann00,BaakeLenzMoody07} for related material). This then holds for arbitrary locally compact, $\sigma$-compact abelian groups.

The implication (ii)$ \Longrightarrow$ (i) is the hard part of the work. It   is shown in \cite{AujoguePhD}  how to  construct a  cut-and-project scheme for $\Del$  under condition (ii). In fact, the construction of \cite{AujoguePhD} even gives that  the associated torus $\TT = (G\times H)/\widetilde{L}$ is just the maximal equicontinuous factor whenever the window $W$ satisfies a suitable `irredundancy' condition.

Under the additional assumptions of unique ergodicity and pure point spectrum, the implication   (ii)$\Longrightarrow$(i) can also directly be inferred  by combining  Theorem 3A and Theorem 6 from  \cite{BaakeLenzMoody07} (and this then holds in  general locally compact $\sigma$-compact  abelian groups).

\subsubsection{Regularity of  the hulls of  FLC Delone sets}
In this section we consider a repetitive  FLC Delone set $\Del$ in $G = \R^N$ with its dynamical system
$(\Omega_\Del,G)$. It turns out that the property of this system to be conjugate to a Meyer dynamical system can be characterized via the maximal equicontinuous factor \cite{KellendonkSadun12}. This characterization provides the crucial additional insight compared to the previous subsection. It allows the derivation of results for FLC Delone sets based on the results for Meyer sets of the last section. As it is (so far) only available for $G =\R^N$ we have to restrict to this situation.
\bigskip

\begin{theorem} \cite{KellendonkSadun12} \label{Characterization-conjugate-to-Meyer} Let $G=\R^N$  and let $\Del$ be a repetitive  FLC Delone set in $G$. The following are equivalent:

\begin{itemize}
 \item[(i)] $(\Omega_\Del,G)$ is conjugate to a Meyer dynamical system.
 \item[(ii)] The system  $(\Omega_\Del,G)$ has at least $N$ linearly independent continuous eigenvalues.

\end{itemize}
\end{theorem}

It is possible to express the result of the previous theorem via the maximal equicontinuous factor.

\begin{cor} Let $G=\R^N$  and let $\Del$ be a repetitive  FLC Delone set in $G$. The following are equivalent:

\begin{itemize}

\item[(i)]  $(\Omega_\Del,G)$ is conjugate to a Meyer dynamical system.

\item[(ii)] The maximal equicontinuous factor of $(\Omega_\Del,G)$ has a factor arising from an action of $G$ on $\R^N / \Z^N$  via the mapping $\R^N \longrightarrow \R^N/\Z^N, x\mapsto  A x + \Z^N$ for some invertible linear map $A : \R^N\longrightarrow \R^N$.

\item[(iii)] The stabilizer of the   $G$-action on the  maximal equicontinuous factor of  $(\Omega_\Del,G)$ is a discrete subgroup.

\end{itemize}

\end{cor}
\begin{proof}  The equivalence between (i) and (ii) is a reformulation of 
Theorem~\ref{Characterization-conjugate-to-Meyer} based on Theorem~\ref{Description-equicontinuous-systems-over-G} and the paragraph preceeding it.

It remains to show the equivalence between (i) and (iii). By Lemma \ref{lem-free}, the condition (iii) is equivalent to  compactness of  $\Gdual /\overline{\Ee_{top}}$. Now, for  $G=\R^N$, compactness of $\Gdual  /\overline{\Ee_{top}}$ can easily be seen to be  equivalent to $\Ee_{top}$ containing  $N$ linear independent vectors. By the previous theorem this is equivalent to (i).
\end{proof}

Combining the results of the previous section with the preceding theorem yields the following equivalences.

\begin{theorem}\cite{AujoguePhD} \label{Meyer-equal-to-mef-one-point} Let $G=\R^N$  and let $\Del$ be  a repetitive FLC-Delone set in $G$.  The following are equivalent:

\begin{itemize}
\item[(i)] $(\Omega_\Del,G)$ is conjugate to a dynamical system  arising from the hull of a complete Meyer set.

\item[(ii)] The dynamical system $(\Omega_\Del,G)$ of $\Del$ is an almost-automorphic system, (i.e.,  the set of points in the maximal equicontinuous factor with exactly  one inverse image point under the factor map is non-empty).
\end{itemize}

\end{theorem}

\begin{theorem}\cite{KellendonkSadun12} \label{Meyer-equal-to-mef-almost-everywhere}
 Let $G=\R^N$  and let $\Del$ be a repetitive FLC Delone set in $G$.  The following assertions are equivalent:

\begin{itemize}
\item[(i)] $(\Omega_\Del,G)$ is conjugate to a dynamical system arising from the hull of a regular complete Meyer set.

\item[(ii)] The dynamical system $(\Omega_\Del,G)$ of $\Del$ is an almost $1$-to-$1$ extension of its maximal equicontinuous factor (i.e., the set of points in the maximal equicontinuous factor with exactly  one inverse image point  under the factor map has full measure).
\end{itemize}
\end{theorem}

Formulation of the analogue of the result on complete periodicity in this context leads to the equivalence of the following two assertions for a repetitive FLC Delone set in $G=\R^N$.

\begin{itemize}
\item $(\Omega_\Del,G)$ is conjugate to a Meyer dynamical system arising from the hull of a completely periodic set.
\item $(\Omega_\Del,G)$ is conjugate to its maximal equicontinuous factor.
\end{itemize}
Now, however, it can easily be seen that a dynamical system coming from a repetitive FLC Delone set is conjugate to that on the hull of a completely periodic set if and only if the original set is already completely periodic. Also, it is not hard to convince oneself that both systems in question are automatically minimal. Thus there is no need to assume repetitivity.
Altogether we then obtain the following result:

\begin{theorem} \label{thm-Delone-FLC} Let $G=\R^N$  and let $\Del$ be a Delone set in $G$ with finite local complexity. The following are equivalent:

\begin{itemize}
\item[(i)] The Delone set  $\Del$ is completely periodic.

\item[(ii)] The dynamical system $(\Omega_\Del,G)$ agrees with its maximal equicontinous factor.
\end{itemize}
\end{theorem}

\textbf{Remark.}  This result had already been proven by Kellendonk / Lenz \cite{KellendonkLenz13} using different methods. In fact, the result of \cite{KellendonkLenz13} is even more general in that it applies to arbitrary compactly generated abelian groups. The result provides an answer to  a question of Lagarias \cite{Lagarias00}.

\section{Proximality}\label{Proximality}
In the previous section we saw the utility of the maximal equicontiuous factor and its factor map. In the present section we present a different approach to this factor by furnishing an alternative description of the equivalence relation defined by $\pi_{max}$. Thus, we will start from a dynamical system $(X,G)$ with compact $X$ and abelian $G$ and the factor map $\pi_{max} : X \longrightarrow \Xmax$.
 In order to ease the presentation of certain concepts we will assume metrizability of the system, i.e., of $X$, in certain places. If the results are valid without this restriction we state them without this restriction (compare also the discussion at the beginning of Section \ref{Equicontinuous}).  Most of the time the systems will furthermore be required to be minimal.

\subsection{Definitions}\label{Definitions}

We consider a variety of relations on $X$. We begin with the relation induced by $\pi_{max}$:
\begin{enumerate}
\item The {\em  equicontinuous
structure relation}
$$R_{max} := \{(x,y)\in X\times X: \pi_{max}(x)=\pi_{max}(y)\}.$$
\end{enumerate}
This relation will be studied by means of the following three relations, each of these is a subset of the equicontinuous structure relation (see the discussion following Theorem \ref{QequalR}).

\begin{enumerate}
\item[2.] The  {\em proximality relation}
$P := \bigcap_{\epsilon>0} P_\epsilon$ with
$$P_\epsilon := \{(x,y)\in X\times X: \mbox{there exists}\;\:  t\in G:d(t\cdot x,t\cdot y) < \epsilon\}.$$

\item[3.] The  {\em regional proximality relation}
 $Q := \bigcap_{\epsilon>0} \overline{P_\epsilon}$, where $\overline{P_\epsilon}$ is the closure of $P_\epsilon$ in the product topology.

\item[4.] The \textit{syndetic proximality relation $s^yP$}, where we say that
    $x$ and $y$ are syndetically proximal if for all $\epsilon>0$,
$$A_\epsilon := \{t\in G: d(t\cdot x,t\cdot y) < \epsilon\}$$
 is relatively dense.
\end{enumerate}
The most intuitive of these relations seems to be the proximality relation because it has a direct dynamical meaning: $x$ and $y$ are proximal if  they can come arbitrarily close when they are moved around with equal group elements. 
We cannot expect $P$ to be transitive. Moreover,
 $P$ need not be a closed relation, i.e., closed in the product topology on $X\times X$.
So while $P$ is intuitive, it can be somewhat tricky.

The regional proximality relation looks like an innocent extension of the proximality relation which is guaranteed to be closed. But care has to be taken! While $Q$ contains $P$ it is in general not the smallest closed equivalence relation containing $P$. It  may be non-trivial even if $P$ is the trivial equivalence relation. For later applications it will be useful to spell out  that $x$ and $y$ are regional proximal if and only if for any $\varepsilon >0$ there exist $(x',y') \in P_\varepsilon$ with $x'$ arbitrarily close to $x$ and $y'$ arbitrarily close to $y$.

The syndetic proximality relation is always transitive, i.e., it is always an equivalence relation. But it is not always closed. It is clearly contained in the proximality relation. Moreover, if the proximality relation is closed then it agrees with the syndetic proximality relation. See \cite {Clay} for proofs of these facts and more information.

One  more thing can be said already about $P$: if the system is equicontinuous then $P$ must be the trivial relation. The converse is not true, however. A point $x\in X$ is called {\em distal} if it is only proximal to itself.
Systems with trivial proximality relation are called \textit{distal}, because they only have distal points. There exist minimal distal systems which are not equicontinuous, but the celebrated theorem of Furstenberg on  the structure of minimal distal systems  (\cite{Furstenberg}) will not concern us here due to the following result. (The result in the stated form is first given in \cite{BargeOlimb11} and is generalized to compactly generated groups in \cite{KellendonkLenz13}, where it forms the core of the proof of Theorem \ref{thm-Delone-FLC}.)

\begin{theorem}\cite{BargeOlimb11,KellendonkLenz13} \label{thm-Olimb}
Consider a repetitive non-periodic FLC Delone set in $\R^N$. There exists two distinct elements of its hull which agree on a half space. In particular the proximality relation on non-periodic FLC Delone systems is non-trivial.
\end{theorem}

\subsection{Some results for general dynamical systems}
We state right away the fundamental result relating  the maximal equicontinuous factor and  proximality. Again we do not attempt to state it in the most general form.
\begin{theorem}[\cite{Auslander88}] \label{QequalR}
Let $(X,G)$ be a   minimal   dynamical system.  Then the equicontinuous structure relation $R_{max}$ is equal to the regional proximal relation $Q$.
\end{theorem}
One direction of containment claimed in this theorem is relatively easy. Note that continuous eigenfunctions of the dynamical system take the same value on proximal pairs: If $x$ is proximal to $y$ and $f$ is a continuous eigenfunction then
$$\frac{f(x)}{f(y)} = \frac{f(t\cdot x)}{f(t\cdot y)} = 1$$
since the first equality is true by the eigenvalue equation for all $t$, and $t\cdot x$ and $t \cdot y$ can get arbitrarily close when varying $t$. A similar argument can be employed if $x$ and $y$ are merely regionally proximal - one only needs to take into account an (arbitrarily small) error of $\epsilon$. 
Thus $Q\subset R_{max}$. For the other direction one needs to prove that the induced action on
$X/Q$ is equicontinuous
which is equivalent to showing that there is an invariant metric generating its
topology.

\medskip

\textbf{Remark.} By the previous theorem, the regional proximality relation opens an alternative way to study the topological spectrum. Given that we have a good intution about proximality, this raises the question: How different are the regional proximal and proximal relations? In general they are quite different. For instance, for a topologically weakly mixing system any two elements are regionally proximal \cite{Auslander88}. But in the minimal case, $x$  can only be proximal to $t\cdot x$ if they are equal. Indeed, suppose that $\inf_{t'\in\R^N} d(t'\cdot x,t'\cdot t\cdot x) = 0$. Then there exists a sequence $(t_n)_n$ such that $\lim_n  d(t_n\cdot x,t_n\cdot (t\cdot x)) = 0$ and so if we take an accumulation point $x'$ of the sequence $(t_n\cdot x)_n$  then
$d(x',t\cdot x')=0$ and so $t\cdot x'=x'$. By minimality we must then also have $t\cdot x=x$. 
\subsubsection{Distal points}\label{Distal-points}
To understand the regional proximal relation it is important  to consider points which are regionally proximal but not proximal.
Recall that a point $x\in X$ is called distal if it is not proximal to
any other point. We denote by $X^{distal}$ the distal points of
$X$. This set might be empty. In fact for metrizable minimal topologically  weakly mixing
systems it is known \cite{Auslander88} (p.\ 132) that $P^2=Q=X\times X$, i.e., for any
two points $x,y\in X$ there exists a third point $z\in X$ such that
$x$ is proximal to $z$ and $z$ is proximal to $y$. This is only
possible if $x$ is not distal. In particular, as we will see that
non-Pisot substitution tilings have weakly mixing dynamical systems, their tiling spaces 
have no distal points. As distality is preserved by the action,
$X^{distal}$ is dense if it is not empty. But for metrizable $X$ even more is true: the set
$X^{distal}$ is even residual if it is not empty due to a remarkable  result of Ellis \cite{Ellis}.
Minimal systems for which $X^{distal}$ is not empty are called \textit{point distal}.
Veech has extended Furstenberg's structure theorem to point distal systems \cite{Veech}.

Let $\xi \in \Xmax$. We call $\pi_{max}^{-1}(\xi )$ the {\em fiber} of $\xi $ without specifying the map $\pmax$.
A point $\xi \in \Xmax$ is called {\em fiber-distal} if all points in its
fiber are distal. Since $P\subset
R_{max}$ a point $x$ can only be proximal to a point in the fiber to
which it belongs, so $\xi $ is fiber-distal if and only if the
proximality relation restricted to its fiber is trivial.
We denote the fiber-distal points by $\Xmax^{distal}$ and say that a
minimal system is fiber-distal if this set is non-empty. Clearly
$$ \pi^{-1}_{max}(\Xmax^{distal})\subset X^{distal}.$$

Let $X^{open}$ be the set of open points for $\pi_{max}$, that is, of
points $x$ such that $\pi_{max}$ maps neighbourhoods of $x$ to
neighbourhoods of $\pi_{max}(x)$.
$X^{open}$ is always a residual set (and hence non-empty) \cite{Veech}.
A point $\xi \in \Xmax$ is called
fiber-open if all points in its fiber $X_\xi  =\pi_{max}^{-1}(\xi )$ are
open. We denote by
$X^{open}_{max}$ the fiber-open points.
We state two fundamental results of Veech \cite{Veech}.
\begin{lemma}[\cite{Veech}]\label{Veech-lemma-one}
Any distal point is open
and so in particular $\Xmax^{distal}\subset X^{open}_{max}$.
Furthermore, $X^{open}_{max}$ is a residual subset of $\Xmax$.
\end{lemma}
In particular $X^{open}_{max}$ is always non-empty  which shows that
the inclusion  $\Xmax^{distal}\subset X^{open}_{max}$
need not be an equality. This is, for instance, the case if the system
has no distal points.
\begin{lemma}[\cite{Veech}] \label{Veech-lemma-two}
If $X^{distal}$ is a residual set then there exists a fiber
$\pi^{-1}_{max}(\xi )$ with a dense set of distal points.
\end{lemma}
\subsubsection{Coincidence rank}
For our further investigation of the various relations we consider three notions of rank for minimal dynamical systems.
The {\em minimal} and the {\em maximal rank}, $mr$ and $Mr$, are the minimal and maximal number of points in a fibre of $\pmax$, or $+\infty$ if the extrema do not exist.
The {\em coincidence rank}, $cr\in\N\cup\{+\infty\}$,  of a minimal dynamical system counts the maximal number of mutually non-proximal points in a fibre. More precisely,
let $\xi \in \Xmax$ and $\mbox{\rm card}(\xi ,\delta)$ be the maximal number $l$ of elements $x_1,\cdots,x_l\in\fb{\xi }$ such that $(x_i,x_j)\notin P_\delta$, or $\mbox{\rm card}(\xi ,\delta)=+\infty$ if this maximum does not exist. There are a couple of observations to make:
First, and this follows from minimality, $\mbox{\rm card}(\xi ,\delta)$ does not depend on $\xi $, and second, $\mbox{\rm card}(\xi ,\delta)$ is decreasing in $\delta$. So we may define the coincidence rank  by
$$ cr = \lim_{\delta\to 0} \mbox{\rm card}(\xi ,\delta)= \sup\{l:\exists x_1,\cdots,x_l\in \pi_{max}^{-1}(\xi ), (x_i,x_j)\notin P\:\mbox{for}\: i\neq j\}$$
and this is independent of $\xi \in \Xmax$.
Here, the first equality is a definition and the  second equality follows easily. Moreover, we  have the following criterion for  finiteness of $cr$.
\begin{lemma} \label{uniform-distance}
The coincidence rank $cr$ is finite if and only if  there exists $\delta_0>0$ such that if $(x_1,x_2)\in Q\backslash P$ then
$d(x_1,x_2)\geq \delta_0$.
\end{lemma}
Indeed, by compactness of the fibers, such a $\delta_0$ cannot exist if $cr$ is infinite, and if $cr$ is finite then the limit $\lim_{\delta\to 0} \mbox{\rm card}(\xi ,\delta)$ must be taken on a strictly positive value for $\delta$.

\smallskip

Independence of $\mbox{\rm card}(\xi ,\delta)$ on $\xi $ also implies that $cr\leq mr$ as the coincidence
may be measured in a fiber of minimal size. So  an almost-automorphic system, i.e., a system with $mr=1$, has also $cr=1$.

The coincidence rank furnishes a criterion for when proximality ($P$) coincides with regional proximality ($Q$).
\begin{theorem}\label{characterization-PequalQ} \cite{BargeKellendonk12}
Let $(X,G)$ be a  minimal system. Then $cr=1$ if and only if $P=Q$.
\end{theorem}
\smallskip

\begin{cor}\label{almost-automorphic-implies-cr-equal-one} Let $(X,G)$ be a minimal system. If $(X,G)$ is almost-automorphic then $cr=1$ and $P = Q$.
\end{cor}

\subsubsection{Consequences of finite coincidence rank}

The following consequence of finite coincidence rank will prove to be useful.
\begin{lemma}\label{lem-cr-free}
Consider a minimal dynamical system with free $\R^N$-action. If the coincidence rank is finite, the $\R^N$-action on its maximal equicontinuous factor is free.
\end{lemma}
\begin{proof} We denote the orbit $\{ t\cdot x: t\in \R^N\}$ of $x$ by $Orb_{\R^N} x$.
Let $H$ be the stabilizer of a fiber $\fb{\xi }$ under the $\R^N$-action. Since the action is free $t x \neq s x$ for $s\neq t$ in $G$ and any $x$ in the dynamical system.  Thus, the cardinality of $H$ is bounded by   the cardinality of $Orb_{\R^N}x\cap \fb{\xi }$ for any  $x\in\fb{\xi }$. Now in the remark after Theorem \ref{QequalR}   we have already discussed  that
$x$ cannot be proximal to a translate of it, as the action is free.
Hence all the points in $Orb_{\R^N}x\cap \fb{\xi }$ are mutually non-proximal. This   means that the cardinality of $Orb_{\R^N}x\cap \fb{\xi }$ is bounded by $cr$. Thus $H$ is finite. Since $\R^N$ has no finite subgroups except the  trivial group, the $\R^N$ action on the maximal equicontinuous  factor must be free.
\end{proof}

Note that if $cr$ is finite, then  a point $\xi \in \Xmax$  is fiber distal if and only if its fiber contains exactly $cr$ elements.

The following yields criteria for when the coincidence rank equals the minimal rank. We need  metrizability of $X$ as we make use of the result of Ellis \cite{Ellis} mentioned above.

\begin{prop} Let $(X,G)$ be a  metrizable minimal system and
suppose $cr<\infty$. The following are equivalent:
\begin{itemize}
\item[(i)] $cr=mr$,
\item[(ii)] $\Xmax^{distal}\neq \emptyset$,
\item[(iii)] the system is point distal ($X^{distal}\neq \emptyset$),
\item[(iv)] $\Xmax^{distal} = X^{open}_{max}$.
\end{itemize}
\end{prop}
\begin{proof}
The equivalence between (i) and  (ii) is rather direct, given that each
fiber must have at least $cr$ elements which are not mutually
proximal.  Clearly (ii)  implies (iii).  We now show that  (iii) implies (i):   If  the system is point distal then, by  the already mentioned result of Ellis \cite{Ellis}, $X^{distal}$ is
residual and hence, by Veech's Lemma \ref{Veech-lemma-two}, there exists $\xi \in \Xmax$ such that
$\pi_{max}^{-1}(\xi )$ contains a dense set of distal points. But if
$cr$ is finite a fiber can only have finitely many distal points
(namely at most $cr$). It follows that the fiber $\pi_{max}^{-1}(\xi )$ is
finite and  contains only distal points. Thus $mr\leq cr$. As the inequality $cr \leq mr$ is clear we obtain $(i)$.

Finally, we discuss the equivalence between (ii) and  (iv).  Here, (ii)  follows from (iv)  by Veech's Lemma \ref{Veech-lemma-one}.  It remains therefore to show that if $\Xmax^{distal}\neq\emptyset$, then a fiber-open point is fiber-distal. Note that $\pi_{max}$ being open at $x$ is equivalent to the condition that whenever $(\xi _n)_n$ is a sequence in $\Xmax$ converging to $\xi  := \pi_{max}(x)$, we can lift that sequence to a sequence $(x_n)_n\subset X$ which converges to $x$.

So let  $\Xmax^{distal}$ be non-empty  and therefore  a dense subset of $\Xmax$ (it is invariant under the action). Let $\xi \in X^{open}_{max}$
and $(\xi _n)_n$ a sequence in $\Xmax^{distal}$ converging to $\xi $. By Lemma \ref{uniform-distance}  the points in a fiber of a fiber distal point have mutual distance at least $\delta_0>0$. Thus,  the set of limits of sequences $(x_n)_n\subset X$ which are convergent and satisfy $\pi_{max}(x_n) = \xi _n$ can have at most $cr$ points. So by the above criterion for open points $\pi_{max}^{-1}(\xi )$ cannot contain more than $cr$ points.
Thus $\xi $ is fiber distal.
\end{proof}
\begin{theorem}\label{Pclosed} \cite{BargeKellendonk12}.
Consider a metrizable minimal system  $(X,G)$ with compactly generated $G$  and connected $\Xmax$. Assume that coincidence rank $cr$ is finite.  Then
$P$ is closed if and only if  $cr=1$.
\end{theorem}

\textbf{Remark.}
The hypothesis that $cr$ is finite is crucial. Indeed, it may happen that $P$ is trivial, and hence closed, whereas $Q$ is not trivial, that is, the system is distal without being equicontinuous. By Theorem \ref{characterization-PequalQ},  such systems do not have $cr =1$. In fact, by the preceding theorem,  such systems
must  always have infinite coincidence rank.

 \medskip

\begin{proof}  (Of Theorem \ref{Pclosed}.)
If $P$ is closed then $X/P$ is a compact space which is metrizable. The maximal equicontinuous factor map $\pmax$ factors therefore as $X\to X/P\stackrel{\pi}\to X/Q=\Xmax$ and the factor map $\pi$ is a $cr$ to $1$ map. Now one needs to show that $\pi$ is a local homeomorphism which in turn is used to show that the system $(X/P,G)$ is equicontinuous.
(This latter result needs that $G$ is compactly generated and $\Xmax$ connected.)  Hence, by maximality of the equicontinuous factor, $X/P=\Xmax$.
\end{proof}
\noindent
\textbf{Remarks.} (a)  For  $G = \R^N$ the assumptions on $G$ and $\Xmax$ are satisfied. Indeed, $\R^N$ is compactly generated. Moreover, $\R^N$ is connected and, due to minimality, $X$, and hence $\Xmax$, must then be connected as well.

(b) By means of suspension of a $\Z^N$-action to an $\R^N$-action (see, for example, \cite{EinWar:11}) the result applies with $G=\Z^N$ as well.

\medskip

\begin{theorem} \label{characterisation-pure-point} \cite{BargeKellendonk12}.
Consider a minimal system  $(X,G)$ with finite coincidence rank $cr$. Suppose
that  $\Xmax^{distal}$ has full Haar measure.
Let $\mu$ be an ergodic probability measure on $X$.
The following are equivalent:
\begin{itemize}
\item[(i)]$cr=1$.
\item[(ii)] The system is an almost $1$-to-$1$ extension of its maximal equicontinous factor.
\item[(iii)] The continuous eigenfunctions generate $L^2(X,\mu)$.
\end{itemize}
If one of these conditions holds then the system is in fact uniquely ergodic.
\end{theorem}
\begin{proof}
The equivalence between (i) and (ii) is rather obvious. Indeed, (ii) $\Longrightarrow$ (i) is clear (compare Corollary \ref{almost-automorphic-implies-cr-equal-one}). The reverse (i)$\Longrightarrow$ (ii)  follows directly as, by its very definition,  the set $\Xmax^{distal}$ consists exactly of those $\xi$ with all points in $\pi_{max}^{-1} (\xi)$ distal, i.e., those $\xi$ with exactly $cr =1$ points in $\pi_{max}^{-1} (\xi)$.
Condition (ii) implies that $\pmax$  yields an isomorphism between $L^2(X,\mu)$ and $L^2(\Xmax,\eta)$ which implies (iii) as $C(\Xmax)$ spans $L^2(\Xmax,\eta)$. To show  (iii)$\Longrightarrow$ (i), i.e., that $\pmax$ can't yield an isomorphism on the level of $L^2$-spaces in case that $cr>1$, we somewhat surprisingly need topology, namely it follows once one has shown that $\pmax$ is almost everywhere a covering map.
\end{proof}

\subsection{Proximality for Delone sets} \label{Proximality-Delone}
We now study particular aspects of proximality for Delone sets in Euclidean space $\R^N$.
The first result concerns the restrictions imposed by
finite coincidence rank 
on a repetitive Delone set.

\begin{theorem}\label{thm-cr-finite}
A repetitive FLC Delone set whose dynamical system has finite coincidence rank is topologically conjugate to a Meyer dynamical system. Moreover, if the Delone set is non-periodic (no periods) then the topological eigenvalues from a dense subgroup of $\hat\R^N$.
\end{theorem}
\begin{proof} The completely periodic case is trivial. We treat the non-periodic case, leaving the case of fewer than $N$ periods to the reader. In this case the $\R^N$-action on the hull is free and hence, by
 Lemma~\ref{lem-cr-free}, so also is the $\R^N$-action on the maximal equicontinuous factor. Lemma~\ref{lem-free} implies now that the eigenvalues are dense and the result follows from Theorem~\ref{Characterization-conjugate-to-Meyer}.
\end{proof}

\subsubsection{Examples and open questions for higher coincidence rank}
In Section~\ref{Hierarchy} we have presented a hierarchy of properties for Delone sets which is based on how large the set of points $\xi\in\Omega_{max}$ is, which have unique pre-image under $\pmax$.
This characterisation concerns the case of minimal rank $mr=1$ (and hence also $cr=1$). We now provide examples of Delone sets whose dynamical system has higher co-incidence rank.

\paragraph{$1<cr<+\infty$} Model sets and periodic sets are now excluded and we know from Thm.~\ref{thm-cr-finite} that the group of eigenvalues is dense. We will see furtherdown that any primitive Meyer substitution tiling has finite maximal rank $Mr$ and hence finite coincidence rank.
Furthermore, such a substitution tiling has $cr = 1$ if and only if its dynamical spectrum is purely discrete (see Theorem~\ref{thm-Meyer-pp}).  
So examples falling into the category $1<cr<+\infty$ are primitive Meyer substitution tilings for which do not have pure point dynamical spectrum. 
The most famous such example comes from the Thue-Morse substitution $0 \mapsto 01$, $1\mapsto 10$. If we suspend this substitution to obtain a substitution tiling in which the length of the $0$ tile and the $1$ tile agree (using a decoration to distinguish them) we obtain a Meyer set which has coincidence rank $cr=2$. 

\paragraph{$cr=+\infty$} Clearly non-trivial Delone systems which are topologically mixing must have infinite coincidence rank, geometrical non-Pisot substitution tilings are of that kind, see below. However there are also examples of Meyer sets which have infinite coincidence rank. For instance the scrambled Fibonacci substitution provides such examples \cite{KellendonkSadun12}. If we take the tiling version in which the tiles have length $\frac{1+\sqrt{5}}2$ and $1$ then the system is topologically weakly mixing, there are no topological eigenvalues besides the trivial one. If we give the tiles both length $1$ then the tiling is Meyer and the topological   eigenvalues form a subgroup of rank $1$. But a subgroup of rank $1$ cannot be dense and thus the coincidence rank must be infinite.

\paragraph{Open questions} Our findings above suggest the following questions:
\begin{enumerate}
\item Does there exists a non-automorphic Delone dynamical system in which the equicontinuous structure relation coincides with proximality? This means that $cr=1<mr$ and thus the system does not have any distal point.
\item More generally, do there exist Delone dynamical systems with finite coincidence rank 
which do not have distal points?  
\end{enumerate}

\subsubsection{Strong proximality and statistical coincidence}
For Delone dynamical systems  (in $\R^N$) it is possible to introduce  stronger
versions of proximality and regional proximality based on the so-called combinatorial metric. Although this metric does not generate the topology of the hull, it is quite useful, particularly for Meyer sets. Recall that two Delone sets are close in the combinatorial metric if they agree on a large ball. Accordingly, it is possible to formulate corresponding stronger versions of proximality without reference to the metric but just via agreement on large balls. To simplify this we define, for a Delone set $\Del$ and an $R>0$,
$$B_R[\Del] := \Del \cap B_R (0).$$
By analogy with the relations 2.,3., and 4. of Section \ref{Definitions} we define the following relations:
\begin{itemize}
\item[5.] Two Delone sets $\Del_1$, $\Del_2$ are \textit{strongly proximal} if for all $R$ there exists $t\in \R^N$ such that $B_R[\Del_1-t] = B_R[\Del_2-t]$.
\item[6.] Two Delone sets $\Del_1$, $\Del_2$ are \textit{strongly regional proximal} if for all $R$ there exist $\Del_1',\Del_2'\in \Omega$, $t\in \R^N$ such that $B_R[\Del_1] = B_R[\Del_1']$, $B_R[\Del_2] = B_R[\Del_2']$ and $B_R[\Del_1'-t] = B_R[\Del_2'-t]$.
\item[7.] Two Delone sets $\Del_1$, $\Del_2$ are \textit{strongly syndetically proximal} if for all $R$ the set of  $t\in \R^N$ for which $B_R[\Del_1-t] = B_R[\Del_2-t]$ is relatively dense.
\end{itemize}

These definitions are indeed strengthenings of the corresponding  relations introduced above.  Obviously, strong syndetical proximality implies strong proximality which in turn implies strong regional proximality.

We comment that strong proximality for Delone sets is even more intutive than proximality:
two Delone sets are strongly proximal if they share arbitrarily large patches. Theorem~\ref{thm-Olimb} shows that strong proximality is also non-trivial for non-periodic repetitive FLC Delone systems.

We note the following consequence: If $\Del_1$ and $\Del_2$ are 
strongly regional proximal elements in the hull $\Omega$ of a Delone set then $\Del_1-\Del_2\subset (\Del-\Del)-(\Del-\Del)$. Indeed, pick $x_1\in\Del_1$ and $x_2\in\Del_2$. Choose $R>\max\{|x_1|,|x_2|\}$. Then there are
$\Del_1',\Del_2'\in \Omega$ and $t\in \R^N$ such that $x_1\in \Del_1'$, $x_2 \in\Del_2'$ and
$t\in \Del_1'\cap\Del_2'$. Set $v_i = t-x_i$. Then $x_1-x_2=v_2-v_1\in (\Del'_2-\Del'_2) - (\Del'_1-\Del'_1)$ and the statement follows as $\Del'-\Del'\subset \Del-\Del$ for each element $\Del'$ in the hull of $\Del$.

\smallskip

We will also consider a statistical variant of the above concepts. This requires a notion of density.   Recall that a sequence $ (\varLambda_n)_{n\in \NN}$ of compact sets in $G$ with
non-empty interior is called {\em van Hove} if it exhausts $G$ and if
\begin{equation*}
   \lim_{n\to \infty} \frac{|\partial^{K} \varLambda_n|}{|\varLambda_n|} \; = \; 0
\end{equation*}
for every compact $K$ in $G$, where for $\Sigma \in G$, we set  $\partial^K \Sigma :=\overline{((\Sigma + K) \setminus \Sigma)}
\cup ((\overline{G\setminus \Sigma} - K)\cap \Sigma)$. Existence of such a sequence can be shown for arbitrary locally compact $\sigma$-compact  abelian $G$ (see, e.g., \cite[p.~145]{Schlottmann00}).
Given such a sequence $\varLambda$, the \textit{upper density} of a subset $B\subset \R^N$ w.r.t.\ the sequence is given by
$$
\overline{dens_\varLambda}(B)=\limsup_n \frac{vol(B\cap\varLambda_n)}{vol(\varLambda_n)}
$$
and the lower density by a similar expression in which the $\limsup$ is replaced by the $\liminf$.
A priori, this notion depends on the choice of sequence.
If both expressions coincide and are independent of the van Hove sequence,
the common value is simply called the density of $B$. In the sequel we will mostly assume that we have fixed a van Hove sequence and suppress dependence on it in the notation.

For a uniformly discrete set $\mathcal{D}$ in $G$, we modify the above to define its upper density:
$$
\overline{dens_\varLambda}(\mathcal{D})=\limsup_n \frac{\sharp(\mathcal{D}\cap\varLambda_n)}{vol(\varLambda_n)},
$$
and for its lower density, replace $\limsup$ by $\liminf$. Which formulas apply when we speak of density will be clear from the context. 
\smallskip

After these preparations we can now introduce the following relation:

\begin{itemize}
\item[8.] Two Delone sets $\Del_1$, $\Del_2$ are \textit{statistically coincident} if
the (upper) density of the symmetric difference $\Del_1\bigtriangleup\Del_2 = (\Del_1\backslash \Del_1\cap\Del_2)\cup (\Del_2\backslash \Del_1\cap\Del_2)$ vanishes. We denote this relation by $SC$.
\end{itemize}

Statistical coincidence implies strong proximality.
\begin{lemma}\label{lem-strong-stat}
Suppose that $\Del_1$ and $\Del_2$ are statistically coincident Delone sets. Then they are strongly proximal.
\end{lemma}
\begin{proof}
If $\Del_1$ and $\Del_2$ are not strongly proximal then there exists $R>0$ such that for all $t\in\R^N$ we have $B_R[\Del_1-t]\neq B_R[\Del_2-t]$. Hence for all $t$ the symmetric difference
$B_R[\Del_1-t]\bigtriangleup B_R[\Del_2-t]$ contains at least one point. It follows that the lower density of
$\Del_1\bigtriangleup\Del_2$ is bounded from below by $1/{\rm vol}(B_R(0))$.
\end{proof}

From this lemma and the  inclusions of the relations discussed above
we immediately have:
\begin{cor}\label{Agreement}
If the statistical coincidence relation $SC$ coincides with the equicontinuous structure relation for a repetitive Delone dynamical system then all relations 1.--7.\ agree with SC  and $cr=1$.
\end{cor}

Let now $(\Omega_\Del,G)$ be the   dynamical system arising from the hull of a Delone set in $\R^N$.  Then Lemma \ref{lem-strong-stat} says that $\pmax$ can be factored as
$\Omega_\Del \to \Omega_\Del/SC \to \Omega_\Del/Q = \Omax$. What is the quotient $\Omega_\Del/SC$? A priori, we do not even know whether $SC$ is a closed relation on $\Omega_\Del$. To investigate this question we consider the so-called autocorrelation hull, following \cite{BaakeLenzMoody07}.

For fixed $r>0$, the mixed autocorrelation pseudometric on the space $\Uu_r$ of all uniformly $r$-discrete subsets of $\R^N$ is given by
$$d_{SC}(\Del_1,\Del_2) = \inf\{\epsilon>0: \exists t_1,t_2\in B(0,\epsilon): \overline{dens}((\Del_1-t_1)\bigtriangleup (\Del_2-t_2))\leq\epsilon\}.
$$
This induces a complete metric on the quotient $\Uu_r/SC$ which we also denote by
$d_{SC}$. Define
$\beta:(\Uu_r,d)\to (\Uu_r/SC,d_{SC})$ by $\beta(\Del) = [\Del]_{SC}$ but mind that the topology on the quotient is, a priori, not the quotient topology. Hence, a priori, $\beta$ is not continuous. To be more clear about this, we write the restriction of $\beta$ to the hull $\Omega_\Del$ of $\Del$ as a composition
$\beta\left|_{\Omega_\Del}\right.=i\circ \mbox{\small id}_{SC}\circ q$,
$$ (\Omega_\Del,d) \stackrel{q}{\to} (\Omega_\Del/SC,\mathcal Q)
\stackrel{\mbox{\small id}_{SC}}{\to}(\Omega_\Del/SC,d_{SC}) \stackrel{i}{\hookrightarrow} (\Uu_r /SC,d_{SC})
$$
where $\mathcal Q$ stands for the quotient topology. So we see that
$\beta\left|_{\Omega_\Del}\right.$ is continuous (which is known to be the case, for example, for a regular complete Meyer set - see Theorem 9 of \cite{BaakeLenzMoody07}) if and only if $\mbox{\small id}_{SC}$ is a homeomorphism; that is, the quotient topology coincides with the metric topology from $d_{SC}$.

The closure of the orbit of $[\Del]_{SC}$ in $\Uu_r/SC$ is called the {\em autocorrelation hull} of the Delone set $\Del$ and is denoted by
$\AM_\Del$. This is not necessarily a compact space and even for repetitive $\Del$ we need to keep track of the dependence on $\Del$ as a locally isomorphic Delone set may, a priori, yield a different autocorrelation hull.

\begin{theorem}\label{thm-SC=Q}
Consider a repetitive Delone set $\Del$. If
$\beta\left|_{\Omega_\Del}\right.$ is continuous, then  $(\AM_\Del,\R^N)$ is isomorphic to
$(\Omax,\R^N)$
and $SC = Q$.
\end{theorem}
\begin{proof}
If
$\beta\left|_{\Omega_\Del}\right.$ is continuous then $\beta(\Omega_\Del) = \AM_\Del$ and so $\AM_\Del$ equals $\Omega_\Del/SC$ and is compact. It follows that there is a factor map $\AM_\Del = \Omega_\Del/SC \to \Omega_\Del/Q = \Omax$. But $d_{SC}$ is invariant under translation and therefore $(\AM_\Del,\R^N)$
equicontinuous. Thus the above factor map must be the identity and $SC = Q$.
\end{proof}
Thus, by Corollary \ref{Agreement}, continuous $\beta\left|_{\Omega_\Del}\right.$ implies $cr=1$. Under  slightly stronger assumptions, namely that $\Del$ is Meyer set and the associated dynamical system is uniquely ergodic, the first statement of the  corollary can also directly  be inferred from  Theorem 7 of \cite{BaakeLenzMoody07} (and the description of the maximal equicontinuous factor via continuous eigenfunctions). That theorem then even implies
that all eigenvalues are topological and the dynamical spectrum is pure point \cite{BaakeLenzMoody07}.

\subsubsection{Strong proximality and the Meyer property}
In this section we apply the theory developed above to  Meyer sets.

\bigskip

\begin{lemma} \label{strong-equal-usual} \cite{BargeKellendonk12}
For repetitive Meyer sets, the strong versions of
proximality, regional proximality and syndetic proximality agree with the usual ones.
\end{lemma}
As a  consequence, repetitive Meyer sets enjoy a stronger form of finite local complexity:
\begin{cor}\label{cor-M1}
Consider a repetitive Meyer set and let $R>0$.
Up to translation there are only finitely many pairs of $R$-patches $(B_R[\Del_1],B_R[\Del_2])$ with
$\pi_{max}(\Del_1)=\pi_{max}(\Del_2)$.
\end{cor}
\begin{proof} By finite local complexity, there are only finitely possibilities, up to translation, for $B_R[\Del_1]$. So the question is: How many relative positions in a pair  $(B_R[\Del_1],B_R[\Del_2])$ may we have? Let $(x_1,x_2)\in (B_R[\Del_1],B_R[\Del_2])$. Since
$\Del_1$ and $\Del_2$ are strongly regional proximal we have $x_1-x_2\in (\Del-\Del)-(\Del-\Del)$.
By the Meyer property the latter set is uniformly discrete; since also $|x_1-x_2|\leq 2R$, we see that we have only finitely many possibilities for $x_1-x_2$.
\end{proof}

\medskip

The notion of coincidence rank becomes more intuitive for repetitive Meyer sets. In fact, if $cr<\infty$  we can combine  Lemma \ref{uniform-distance}
and Lemma \ref{strong-equal-usual} to obtain that
 there exists   $R_0>0$ such that
\begin{equation*}
cr = \max\{l | \exists \Del_1,\cdots,\Del_l\in\pi_{max}^{-1}(\xi ): \forall t\in \R^N,
B_R[\Del_i-t]\neq  B_R[\Del_j-t]\mbox{ for }i\neq j\}
\end{equation*}
where $\xi \in \Omax$ and $R\geq R_0$ are arbitrary. This has two interpretations. A priori, one could expect that the maximum on the r.h.s.\ becomes larger if $R$ gets larger, because disagreement on larger patches is a weaker condition than on smaller patches. But this is not the case as soon as $R$ is larger than a certain threshold value $R_0$. A second interpretation is that the distinct Meyer sets in the fiber of a fiber distal $\xi $ have at each point $t\in\R^N$ distinct $R_0$-patches; that is, they are {\em non-coincident} on $R_0$-patches.
If $cr$ is not finite then the max on the r.h.s.\ becomes indeed arbitrarily large as $R$ tends to infinity.

Let $n^R(\xi )$ be the number of different $R$-patches at $0$ which occur in the elements of the fibre of $\xi $, i.e., 
$$ n^R(\xi ) = \#\{B_R[\Del]:\Del\in\fb{\xi }\}.$$
For Meyer sets, this number is finite by (the proof of) Corollary.~\ref{cor-M1} and we derive from the preceding formula for $cr$  that $cr\leq n^R(\xi)$,  provided $cr$ is finite and $R$ is large enough ($R\geq R_0$).
Note that $\lim_{R\to\infty} n^R(\xi )$ is the cardinality of $\fb{\xi }$.
The following lemma was stated and proved in \cite{BargeKellendonk12} under the assumption of finite maximal rank. This assumption can in fact be dropped.
\begin{lemma}\label{lem-upper}
For repetitive Meyer sets, $n^R$ is upper semi-continuous; that is, the sets $\{\xi :n^R(\xi )\geq k\}$ are closed in $\Omax$ for any $k\geq 0$.
\end{lemma}
\begin{proof} Note that
if $\mathcal{L}_n\to \mathcal{L}$ in ${\Omega_\Del}$ then there exists a sequence $(t_n)_n\subset \R^N$, $t_n\to 0$, such that
$B_R[\mathcal{L}_n-t_n] = B_R[\mathcal{L}]$ for all sufficiently large $n$. Now fix $k\in\N$
and suppose that $(\xi _n)_n$ is a sequence in $\Omax$ with $n^R(\xi _n)\ge k$ for all $n$ and with $\xi _n\to \xi $. It follows from Corollary~\ref{cor-M1} that there is $\delta>0$ with the property that if $\mathcal{L},\mathcal{L}'\in {\Omega_\Del}$ are such that $\pi_{max}(\mathcal{L})=\pi_{max}(\mathcal{L}')$ and $B_R[\mathcal{L}]\ne B_R[\mathcal{L}']$, then
$B_R[\mathcal{L}-t]\ne B_R[\mathcal{L}'-t'])$ holds true even for all $t,t'\in B_\delta(0)$. Thus, for each $n$, we may choose $\mathcal{L}_n^1,\ldots,\mathcal{L}_n^k\in\fb{\xi_n}$ with
$B_R[\mathcal{L}_n^i-t]\ne B_R[\mathcal{L}_n^j-t'])$ for $i\ne j$ and $t,t'\in B_\delta(0)$.
By passing to a subsequence, we may assume that $\mathcal{L}_n^i\to\mathcal{L}^i\in\pi_{max}^{-1}(\xi )$ for $i=1,\ldots,k$. Then
$B_R[\mathcal{L}^i]\ne B_R[\mathcal{L}^j]$ for $i\ne j$,
so $n^R(\xi )\ge k$ and the set $\{\xi :n^R(\xi )\ge k\}$ is closed.
\end{proof}

\begin{theorem}\label{thm-cr-stat}
Consider a repetitive Meyer set $\Del$  whose  dynamical system $(\Omega_{\Del},\R^N)$ has finite coincidence rank.
If the fiber distal points have full Haar measure then, for any $R\geq R_0$ and any $\xi \in\Omax$,
there is a subset $A\subset\R^N$ of density $1$
such that $$ \#\{B_R[\Del-t]:\pmax(\Del) = \xi \} = cr\quad \mbox{for all } t\in A.$$
\end{theorem}

\textbf{Remark.} We can summarise the theorem as  saying that locally (that is, by inspection of finite patches) and with probability $1$ all fibers have $cr$ elements. Consider for instance the system associated with the Thue-Morse substitution, which has coincidence rank 2 (see the chapter on the Pisot Substitution Conjecture in this volume). The maximal equicontinuous factor has one orbit whose fibres have $4$ elements. But only near the `branching locus' can one find $4$ different $R$-patches, otherwise there are only $2$.

\begin{proof}
Let
$D^R = \{\xi\in\Omax: n^R(\xi) = cr\}$. Then $\Omax^{distal} = \bigcap_{R\geq R_0} D^R$ and so
the hypothesis implies that $\eta(D^R) = 1$. Here, $\eta$ denotes the Haar measure on the maximal equicontinuous factor.  Now let
$\tilde n^R = n^R - cr 1_{D^R}$; that is, $\tilde n^R$ is $0$ on $D^R$ and otherwise the same as $n^R$. By the preceding lemma, $n^R$ is upper semi-continuous.  Hence, $\tilde n^R$ is a finite positive linear combination of indicator functions on compact sets.
 As the maximal equicontinuous factor is uniquely ergodic, we obtain then a uniform inequality in the  ergodic theorem (see e.g. Lemma 4 in  \cite{Lenz09}). More specifically, we have  for all $\xi $ and all $R\geq R_0$
\begin{equation*}
\limsup_n \frac1{vol(\varLambda_n)} \int_{\varLambda_n}\tilde n^R(t\cdot \xi ) dt
\leq \int \tilde n^R(\xi ')d\eta(\xi '),
\end{equation*}
where $\varLambda=(\varLambda_n)_n$ is a van Hove sequence for $\R^N$.
Due to  $\eta(D^R)=1$ the right hand side in the previous inequality is  $0$.
Let $B = \{t\in \R^N : n^R(t\cdot \xi ) \neq cr\}$. Then $1_B(t)\leq \tilde n^R(t\cdot \xi )$ and thus
$$ \overline{dens_\varLambda}(B) =\limsup_n \frac1{vol(\varLambda_n)} \int_{\varLambda_n}1_B(t)dt
\leq
\limsup_n \frac1{vol(\varLambda_n)} \int_{\varLambda_n}\tilde n^R(t\cdot \xi )dt = 0.$$
Hence the density of $B$ is $0$ and $A=B^c$, the complement of $B$, has the required property.
\end{proof}
\begin{cor}
Let $\Del_1$ and $\Del_2$ be two elements in the hull of a regular complete Meyer set. 
Then $\pi_{\max} (\Del_1) = \pi_{\max} (\Del_2)$ if and only if  they are statistically coincident.
\end{cor}
\begin{proof}
By assumption, the hypothesis of the last theorem is satisfied with $cr=1$. In particular $Q$ agrees with  the strong proximality  relation and so one direction follows immediately from Lemma~\ref{lem-strong-stat}. It remains to show that if $\Del_1$ and $\Del_2$ belong to the same fibre of $\pmax$ then they are statistically coincident. But the last theorem just says that in this case the density of points where $\Del_1$ and $\Del_2$ agree on an $R$-ball is $1$, or, in other words, the density of points where $\Del_1$ and $\Del_2$ disagree is $0$. Hence $\Del_1$ and $\Del_2$ are statistically coincident.
\end{proof}

\textbf{Remark.} The statement of the  above corollary is quite at the heart of the considerations of \cite{BaakeLenzMoody07}. In fact, as mentioned above, 
Theorem 9 of \cite{BaakeLenzMoody07} gives
that, for regular complete Meyer sets, the map
$\beta\left|_{\Omega_\Del}\right.$ is indeed continuous. When combined with Theorem 7 of \cite{BaakeLenzMoody07}, we  directly obtain the statement of the corollary.
This approach actually  shows that the result is valid not only in $\R^N$ but for general locally compact $\sigma$-compact  abelian groups.

\subsubsection{Meyer substitutions}\label{Meyer substitutions}
So far we have formulated all our results  for Delone sets rather than tilings. This does not really make a difference, as the two are related by mutually local derivations. In particular, all the concepts and results translate into the formalism of tilings.
In this section we will use the formalism of tilings,
because we find it much more convenient and intuitive for substitutions.
For us, a {\em tile} in $\R^N$ is a subset homeomorphic to a compact $N$-ball and a {\em tiling} of $\R^N$ is a collection of tiles with disjoint interiors which covers $\R^N$.
The set of tiles of a tiling $T$ which intersect non-trivially a compact subset $K\subset\R^N$ is called a {\em patch}. In particular, the $R$-patch at $0\in\R^N$ is the set of tiles which touch the closed $R$-ball $B_R(0)$; we denote it by $B_R[T]$.
The {\em support} of a patch $P$, $\supp(P)$, is the set of points covered by the tiles of $P$.\footnote{One could include the possibility of decorating the tiles in case one wants to distinguish translationally congruent tiles and then distinguish the support of a tile (the points covered by the tile) from the decorated tile.}

For a tiling substitution we suppose we have a finite set $\mathcal{A}=\{\rho_1,\ldots,\rho_k\}$ of
translationally inequivalent tiles (called {\em prototiles}) in $\R^N$ and an expanding linear map
$\Lambda$.
A {\em substitution} on $\mathcal{A}$ with expansion $\Lambda$ is a function $\Phi:\mathcal{A}\to\{P:P$ is a patch in $\R^N\}$
with the properties: for each $i\in\{1,\ldots,k\}$, every tile in $\Phi(\rho_i)$ is a translate of an element of $\mathcal{A}$; and $\supp(\Phi(\rho_i))=\Lambda(\supp(\{\rho_i\}))$. Such a substitution naturally extends to patches and even tilings whose elements are translates of the prototiles and it satisfies
$\Phi(P-t) = \Phi(P)-\Lambda(t)$.

A patch $P$ is {\em allowed} for $\Phi$ if there is an $m\ge1$, an $i\in\{1,\ldots,k\}$, and a $v\in\R^N$, with $P\subset \Phi^m(\rho_i)-v$. The {\em substitution tiling space} associated with
$\Phi$ is the collection $\OP$ of all tilings $T$ of $\R^N$ such that every finite patch in $T$ is allowed for $\Phi$. $\OP$ is not empty and, since translation preserves allowed patches, $\R^N$ acts on it by translation. To define a metric on $\Omega_{\Phi}$, we can borrow the metric we've used for Delone sets: Pick a point $y_i$ in the interior of each prototile $\rho_i$ and, for $T\in\Omega_{\Phi}$, let $\Del(T)=\{y_i+x:x+\rho_i\in T\}$. Then set $d(T,T'):=d(\Del(T),\Del(T'))$. (The set $\Del(T)$ is called a {\em set of punctures} of $T$.)

The substitution $\Phi$ is {\em primitive} if for each pair $\{\rho_i,\rho_j\}$ of prototiles there is a $k\in\N$ so that a translate of $\rho_i$ occurs in $\Phi^k(\rho_j)$.  If the translation action on $\OP$ is free, which is equivalent to saying that each of its elements is a non-periodic tiling, then $\Phi$ is said to be {\em non-periodic}. If all tilings from $\OP$ are FLC then
$\Phi$ is said to be FLC. If $\Phi$ is primitive, FLC and non-periodic then $\OP$ is compact in the metric described above,
 $\Phi:\OP\to\OP$ is a homeomorphism, and the translation action on $\OP$ is minimal and uniquely ergodic ( \cite{AP}, \cite{Solomyak99}, \cite{sol}). In particular, $\OP=\Omega_T:=\overline{\{T-v:v\in\R^N\}}$ for any $T\in\OP$. It will be with respect to the unique ergodic measure $\mu$ on $\OP$ when we speak about the dynamical spectrum and $L^2$-eigenfunctions.
 In the context of eigenfunctions (non-periodic) substitutions have a rather special feature: All measurable eigenfunctions are continuous. Thus, all eigenvalues are automatically continuous eigenvalues. For symbolic dynamics this result  is due to Host \cite{Host86}. The case at hand is treated by Solomyak \cite{Solomyak07}.

\smallskip

A Meyer substitution is a substitution $\Phi$ such that the elements of $\OP$ are Meyer tilings, that is, they are
MLD to a Meyer set. To check this it suffices to check that, for $T\in\OP$, the set of punctures $\Del(T)$ is a Meyer set.

It is easily verified that $\Phi$ preserves the regional proximality relation and therefore induces a homeomorphism $\Phi_{max}$ on the maximal equicontinuous factor $\Omax$. In particular $\Phi_{max}$ satisfies a similar equation
$$\Phi_{max}(\xi -t) = \Phi_{max}(\xi )-\Lambda(t)$$
from which one concludes, as $\Lambda$ has no root of unity eigenvalues, that $\Phi_{max}$ is ergodic w.r.t.\ Haar measure.

\begin{prop}[\cite{BargeKellendonk12}]\label{thm-Meyer}
Consider a Meyer substitution tiling system with primitive non-periodic substitution. Then the following hold:
\begin{enumerate}
\item[(a)] The maximal rank is finite.
\item[(b)]  $\Omax^{fiber}$ has full measure.
\item[(c)] Syndetic proximality is a closed equivalence relation.
\item[(d)] Two distinct tilings of a fiber distal fiber do not share a common tile.
\end{enumerate}
\end{prop}
\begin{proof}
We indicate the idea of some of the proofs. Since $\Lambda$ is expanding there is a $c>1$ and an $n\in\N$ such that $B_{cR}(0)\subset \Lambda^n(B_R(0))$. Replacing $\Phi$ by $\Phi^n$, we may suppose that $n=1$. Hence
$$ n^{cR}(\Phimax(\xi)) = \#\{B_{cR}[\Phi(T)]:\pmax(T) =\xi\} \leq \#\{\Phi(B_R[T]):\pmax(T) =\xi\} \leq n^R(\xi).$$
From this we see that the maximal rank is bounded by $\sup_{\xi\in\Omax}n^0(\xi)$ which is finite by
Corollary~\ref{cor-M1}. The argument for the fourth statement is based on a similar reasoning.

Since $n^{R}\leq n^{cR}$ the above shows also that the sets $D^R(m):= \{\xi\in\Omax:n^R(\xi)\leq m\}$, $m\in\N$, are invariant under $\Phimax$. By Lemma~\ref{lem-upper} $D^R(m)$ is open.
By ergodicity of $\Phimax$ therefore, it has measure $1$, provided it is not empty.
Since $\Omax^{distal} = \bigcap_{R\geq R_0} D^R(cr)$ the second statement follows if we show that $D^R(cr)\neq\emptyset$.

Consider a fiber which has minimal rank, i.e., $\xi\in \Omax$ such that
$\fb{\xi} = \{T_1,\cdots,T_{mr}\}$.
Suppose that for all $r>0$ there exists $w\in\R^N$ such that $\forall t\in B_r(w)$ we have $n^R(\xi-t)\geq mr$; that is, all $B_R[T_i-t]$, $1\leq i\leq mr$, are distinct. Then we can find two sequences $(r_k)_k \to \infty$ and $(w_k)_k\in\R^N$ such that $(T_i - w_k)_k$ converge in $\Omega$, let's say to $S_i$, and
$(\xi - w_k)_k$ converges in $\Omax$, to $\zeta$, say, and $\forall t\in B_{r_k}(0)$  all
$B_R[T_i-w_k-t]$, $1\leq i\leq mr$, are distinct. Taking $k\to \infty$ we conclude that all
$B_R[S_i-t]$, $1\leq i\leq mr$, $t\in\R^N$, are distinct. In particular, the $S_i$ belong to the fiber of $\zeta$ and are pairwise non-proximal and so $cr\geq mr$. This shows that $cr=mr$ and hence $D^R(cr)$ is not empty.

It remains to argue that our assumption is satisfied. So let us suppose the contrary, namely that there exists $r>0$ such that for all $w\in\R^N$ there exists $t\in B_r(w)$ with
$n^R(\xi-t)\leq mr-1$. It follows that the lower density of points $t\in R^n$ such that $n^R(\xi-t)\leq mr-1$ is strictly positive. Since for all $t$ we have that $n^R(\xi-t)\leq mr$ ($\xi$ lies in a fiber of rank $mr$)  the ergodic theorem implies that $\int_{\Omax} n^R(\xi) d\eta(\xi) < mr$. Hence $D^R(mr-1)$ can't have measure $0$. So it must have measure $1$. But then $\bigcap_{R\geq R_0} D^R(mr-1)$
has measure $1$ and so
there must be a fiber of rank at most $mr-1$ which contradicts the minimality of $mr$. This shows the second statement.

For the third statement see \cite{BargeKellendonk12}.
\end{proof}
As a consequence of the previous  proposition  and our earlier results (in particular, Theorem \ref{characterisation-pure-point} and Corollary \ref{Agreement}) we have the following list of equivalent characterizations of pure point (measure) dynamical spectrum which hold for primitive non-periodic Meyer substitutions.
\begin{theorem}\label{thm-Meyer-pp}
Consider the system of a primitive non-periodic Meyer substitution. The following are equivalent:
\begin{itemize}
\item[(i)] The (measure) dynamical spectrum is purely discrete.
\item[(ii)] The dynamical system is an almost $1$-to-$1$ extension of its maximal equicontinuous factor.
\item[(iii)] The coincidence rank $cr$ is $1$.
\item[(iv)] The (strong) proximality relation is closed.
\item[(v)] The (strong) proximality relation coincides with the equicontinuous structure relation $R_{max}$.
\item[(vi)] The (strong) proximality relation coincides with the (strong) syndetic proximality relation.
\end{itemize}
\end{theorem}
We finally present a result of Lee and Solomyak which, for a particular class of substitution tilings, characterizes those which are Meyer substitutions. This class is defined by some further conditions\footnote{It has been announced in \cite{Kwapisz:2013} that these conditions can be considerably weakened: see the discussion of Pisot families in the chapter on the Pisot Substitution Conjecture in this volume for more detail.} on the linear expansion $\Lambda$, namely that
\begin{enumerate}
\item $\Lambda$ is diagonalizable (over $\C$), 
\item all eigenvalues of $\Lambda$ are algebraically conjugate,
\item  all  eigenvalues of $\Lambda$ have all the same multiplicity.
\end{enumerate}
It should be said that the eigenvalues of the linear expansion are algebraic integers (\cite{K2},\cite{LS2}) and, if the expansion is diagonalizable they
form a union of families (\cite{KS}).
Here, a family is the set $F_{p,c}$ of roots of a monic, irreducible, integer polynomial $p$ which have absolute value greater than some real number $c>0$. In other words, a family is a subset of the set of algebraic conjugates of some algebraic integer which can be characterized by the property that if it contains $\lambda$ then it contains all conjugates which have absolute value at least as large as that of $\lambda$.
The family $F_{p,c}$ is called a {\em Pisot family} if $c=1$ and $p$ has no roots of
absolute value $1$.
\begin{theorem}[\cite{LS}]\label{thm-Pisot-Meyer}
Consider a primitive non-periodic $N$-dimensional FLC substitution with expansion $\Lambda$ satisfying the above three properties. The following are equivalent:
\begin{itemize}
\item[(i)] The substitution is Meyer
\item[(ii)] The eigenvalues of $\Lambda$ form a Pisot family.
\item[(iii)] The continuous eigenvalues of the $\R^N$-action on the hull are dense in $\hat{\R}^N$. 
\item[(iv)] The maximal equicontinuous factor is non-trivial.
\end{itemize}
\end{theorem}
Recall from the discussion at the beginning of Section \ref{Distal-points} that triviality of the maximal equicontinuous factor implies absence of distal points.  Given this, we can combine the previous theorem with  Proposition~\ref{thm-Meyer} and
Theorem~\ref{thm-cr-finite} to obtain the following  strong  dichotomy.

\begin{cor} Consider a primitive non-periodic FLC substitution with expansion $\Lambda$ satisfying the above three properties. Either the system has no distal points, or  the distal points form a set of full measure. In the first case
the dynamical point  spectrum is trivial and in the second the continuous eigenvalues are dense.
\end{cor}

As a consequence of the above discussion we also obtain the following remarkable statement: When it comes to the question of which point sets or tilings have pure point spectrum, all examples produced by  substitutions  could also be obtained by the cut-and-project formalism. More specifically, the following holds.

\begin{cor}\label{subst-complete-Meyer} Suppose that $\Phi$ is a primitive non-periodic $N$-dimensional FLC substitution with expansion $\Lambda$ satisfying the above three properties. Let $T\in\Omega_{\Phi}$ and let $\Del (T)$ be a set of punctures. Then $(\Omega_{\Phi},\R^N)$ has pure point spectrum if and only if $\Del (T)$ is a regular complete Meyer set.
\end{cor}
\begin{proof} Let $(\Omega_{\Phi},\R^N)$ have pure point spectrum. As all eigenvalues are continuous, we infer that the maximal equicontinuous factor is non-trivial.
By Theorem \ref{thm-Pisot-Meyer} the substitution must then be  Meyer. Hence,
Theorem~\ref{thm-Meyer-pp} gives that  the dynamical system is an almost $1:1$ extension of its maximal equicontinuous factor. Now, Theorem
\ref{Meyer-equal-to-mef-almost-everywhere} implies that $\Del (T)$ is a regular complete Meyer set.

\smallskip

As for the converse direction, we note that any regular Meyer set gives rise to a dynamical system which is  an almost $1:1$ extension of it's maximal equicontinuous factor by Theorem \ref{Meyer-equal-to-mef-almost-everywhere}.  From Theorem \ref{thm-Meyer-pp} we then infer pure point spectrum.
\end{proof}

\textbf{Remark.}  Of course, it is well known that a regular complete Meyer set gives rise to a dynamical system with pure point spectrum  (see  e.g. discussion in Section \ref{Background-Meyer}).   The main part of the corollary is thus the converse implication.  It has been shown for one-dimensional systems by Sing \cite{Sing}. For  higher dimensional self-similar substitutions is has been obtained by Lee in \cite{Lee}. Note, however, that the  work of Lee does not seem to claim regularity of the Meyer set but just its completeness. See also the chapter on the Pisot Substitution Conjecture in this volume.
\newpage


\section{Ellis semigroup}\label{Ellis}
If the action of a group $G$ on a space $X$ is transitive we can view $X$ and its maximal equicontinuous factor $\Xmax$ as two distinct compactifications of the acting group $G$, the difference arising from the topology in which it is compactified.
In this section we consider a third compactification of $G$ -- the Ellis semigroup
$E(X,G)$ of the dynamical system $(X,G)$.
As a space and dynamical system it tends to be by far the most complicated of the three compactifications.
But it has one advantage; namely, it naturally carries the structure of a monoid (i.e., a semigroup with neutral element). It therefore offers the possibility to characterize dynamical systems by means of this algebraic structure. There are  only a few types of systems for which this has been successfully carried out; non-periodic Delone systems are, however, not among these. So as a first step we
simply present some explicit examples of Ellis semigroups coming from Delone sets and observe that they exhibit a very interesting algebraic structure, which we have not seen before in this context. The examples we present are associated with almost canonical cut-and-project patterns. Almost canonical cut-and-project patterns are complete Meyer sets whose windows are polyhedral satisfying further conditions. Surprisingly, the Ellis semigroup for such dynamical systems has a very particular form. 
It is a completely regular semigroup (or a union of groups) \cite{Petrich}. Furthermore, its idempotents from a submonoid which is reminiscent of the so-called face semigroup associated to a hyperplane arrangement \cite{Brown}.

The material of Section~\ref{sec-Ellis1} is mostly based on the book of Auslander \cite{Auslander88}, although most of it can also be found in the book of Ellis \cite{Ellisbook},
 and that of the later sections in the thesis of the first author \cite{AujoguePhD,Aujogue13}.
\subsection{Definitions and known properties}\label{sec-Ellis1}
An action $\alpha$ of a group $G$ on the compact space $X$ is nothing else than a representation of the group in terms of transformations of $X$; i.e., for each $t\in G$, $\alpha_t$ is a function from $X$ to $X$. The set of all functions from $X$ to $X$ is the product set $X^X$ and becomes a compact space when equipped with the Tychonoff topology.
\begin{definition}
The Ellis semigroup $E(X,G)$ is the closure of $\{\alpha_t | t\in G\}$ in $X^X$.
\end{definition}
By definition of the Tychonoff topology, a net $(f_\lambda)_\lambda$ of functions $f_\lambda:X\to X$
converges to some function $f:X\to X$ if for all $x\in X$
the net $(f_\lambda(x))_\lambda$ converges to $f(x)$. Any element of $E(X,G)$ is thus a limit of a net $(\alpha_{t_\lambda})_\lambda$ where $(t_\lambda)_\lambda$ is a net in $G$.
Assuming that the action is non-degenerate we may identify
$G$ with $\{\alpha_t | t\in G\}$ and see that $E(X,G)$ is indeed a
compactification of $G$.
Furthermore,
$G$ acts on $E$ from the left: $\alpha_t^E(f) = \alpha_t\circ f $.
Thus $(E(X,G),G)$ is a dynamical system.
A factor map $\pi:(X,G)\to (Y,G)$ induces a continuous surjective monoid morphism
$\pi_*:E(X,G)\to E(Y,G)$. In fact, the latter is given  by the equality
$\pi_*(f)(\pi(x)) = \pi(f(x))$ \cite{Auslander88}[Thm.~7, p.~54].

The basic idea is now to characterize the dynamical system $(X,G)$ by means of the
properties of $E=E(X,G)$. We may   ask the following questions:
\begin{itemize}
\item
$E$ consists of functions $f:X\to X$. What are their properties?
\item What is the algebraic structure of $E$?
\item
$E$ is a compact Hausdorff space.
What more can we say about its topology?
\end{itemize}
Let us elaborate.

\smallskip

$E$ consists of functions $f:X\to X$ which are limits of homeomorphisms. Are these functions still homeomorphisms? If not, are they at least continuous or invertible?
We provide an elementary argument why this cannot be the case for all elements in the semigroup
of  the dynamical system of a repetitive non-periodic FLC Delone set.
Recall from Theorem ~\ref{thm-Olimb} that there are two distinct Delone sets $\Del^+,\Del^-$ in the hull which agree on a half space. We choose coordinates such that the first component corresponds to the normal into that half space. Then, whenever  $(t_\Lambda)_\lambda$ is a net such that the first component of $t_\lambda$ tends to $+\infty$, $\Del^+ - t_\lambda $ and $\Del^- - t_\lambda$ agree on larger and larger balls. By repetitivity we may assume that there are two such nets for which
$\lim_\lambda \Del^\pm-t^\pm_\lambda = \Del^\pm$. By compactness of $E$ we may assume that
$\alpha_{t^{-}_\lambda}$ converges to an element $f\in E$. Then
$f(\Del^+)=\lim_\lambda ( \Del^+-t^-_\lambda) = \Del^-$, because eventually the sets
$\Del^+-t^-_\lambda$ and $\Del^--t^-_\lambda$ come close. On the other hand $f( \Del^+-t^+_\lambda) =  f(\Del^+)-t^+_\lambda = \Del^--t^+_\lambda$ which tends to $\Del^+$.
Hence if $f$ were continuous the last argument would give $f(\Del^+) =\Del^+$, a contradiction.
So $f$ is not continuous. Furthermore, $f(\Del^-)=\lim_\lambda\Del^--t^-_\lambda = \Del^-$ and hence $f$ is not injective.

The set $E$ is closed under composition of functions. Accordingly,
composition of functions makes $E$ into a semigroup.
Moreover,  the identity $\alpha_0$ is a unit for the composition law
and so $E$ is a monoid.
Care has to be taken, however, concerning the topological properties of the semigroup product.
If $(f_\lambda(x))_\lambda$ converges to $f(x)$ then $\lim_\lambda f_\lambda\circ g = f\circ g$ and so the semigroup product is continuous in the left variable: stated differently, right translation by $g$, $f\mapsto fg:=f\circ g$ is a continuous map. It is however, in general, not true that the semigroup product is continuous in the right variable. Moreover, even in case that $G$ is abelian
(which is the only case that will concern us)
the semigroup product in $E(X,G)$ is, in general, not commutative (although any
element of $E(X,G)$ commutes with elements coming from $G$, i.e.,
elements of the type $\alpha_t$).

From the point of view of topology, the Ellis semigroup is either well behaved in the sense that it is separable
and every element is the limit of a sequence (as opposed to net) $(\alpha_{t_n})_n$, $t_n\in G$,  or it is rather wild in that it contains a homeomorphic image of the Stone-Cech compactification of $\N$ \cite{GlasnerMegrelishvili06}.
In the first case the semigroup is called {\em tame}. If the semigroup is first countable, i.e., has a countable neighborhood base, then it is tame.
\bigskip

One of the important results of the general theory concerns  minimal left ideals and the existence of idempotents, that is, elements $p$ satisfying $p^2=p$.
Let $F$ be a closed $G$-invariant subset of $E=E(X,G)$. Then $\alpha_t^E(F) \subset F$ implies that
$F$ is a left ideal of $E$. It follows also that any minimal left ideal is a minimal component of the dynamical system and hence closed. Furthermore, an application of Zorn's lemma yields that
every (closed) minimal left ideal of $E$ contains idempotents.

\smallskip

There are several connections between proximality and the Ellis semi-group. Two that are fundamental are highlighted in the following theorem.

\begin{theorem}
Consider a compact minimal dynamical system $(X,G)$.
\begin{enumerate}
\item[(a)] $x\in X$ and $y\in X$ are proximal if and only if $p(x) = p(y)$ for some minimal idempotent of $E(X,G)$ (a minimal idempotent is an idempotent in a minimal ideal)
\cite{Auslander88}[Thm.~13, p.~89].
\item[(b)] Proximality is transitive if and only if $E$ contains a unique minimal ideal \cite{Auslander88}[Thm.~10, p.~88].
\end{enumerate}
\end{theorem}
Let us note that part (b) of the previous theorem gives uniqueness of the minimal ideal whenever  $ P  = R_{\max}$. In particular, uniqueness of the mimimal ideal holds   for minimal systems if $cr =1$  (see the discussion in Section \ref{Proximality}).

\smallskip

The somewhat nicest case is that in which the system $(X,G)$ is equicontinuous.

\begin{theorem}[\cite{Auslander88} Thm.~3 \& 5, p.~52,53] \label{Ellis-fundamental}
Consider a dynamical system $(X,G)$. The system is equicontinuous if and only if $E(X,G)$ is a group and all its elements are homeomorphisms. If, moreover, the action is minimal then
$(E(X,G),G)$ is topologically conjugate to $(\Xmax,G)$ and the conjugacy is a group isomorphism.
\end{theorem}
Much more can be said to support the statement: If the semigroup $E$ is well-behaved then the system is close to being equicontinuous. For instance, if, for a minimal system, all elements of $E$ are continuous, then they are even homeomorphisms and the system is equicontinuous \cite{EllisNerurkar89}.
On the other hand, if all elements of $E$ are bijective and so $E$ is a group, then the system must be distal and, conversely, triviality of the proximal relation implies that $E$ is a group. Finally we mention that (again for minimal systems)  the topology of $E$ is metrizable if and only if the system is  equicontinuous \cite{Glasner06}.
For a non-periodic FLC Delone system, however, the  Ellis semigroup is neither a group nor is it metrizable.

\bigskip

Let $M$ be a minimal left ideal and $J$ the set of its idempotents. Then for every $p\in J$ and $m\in M$ we have $mp = m$ and, furthermore,  the restriction of the semigroup product to $pM$ makes $pM$ a group with neutral element $p$ \cite{Auslander88}[Lemma~1, p.\ 83].
Moreover, all the groups $pM$ are isomorphic, the isomorphism between $pM$ and $qM$ being given by $pm\mapsto qm$; and $M$ is their disjoint union: $M = \bigcup_{p\in J} pM$.
 Let $\mathcal G := p_0M$ for some chosen $p_0\in J$. Then, as a semigroup,
 $$ M \cong \mathcal G\times J$$
where we take the product operation on the r.h.s..
The semigroup isomorphism is given by $pm \mapsto (p_0 m,p)$. Indeed,
$J$ is a sub-semigroup with product given by the so-to-say left domination rule
$$ pq = p,\quad p,q\in J$$
and $pmp'm' = pmm'$, showing that the above map preserves the semigroup product.
We can say a little bit more about the group $\mathcal G$. Since
 $p_0Ep_0\subset p_0M$ and $(\pmax)_*(p_0) = \id$ ($p_0$ is an idempotent and $\id$ is neutral element in $E(\Xmax,G)$), $(\pmax)_*$ restricts to a surjective group homomorphism $(\pmax)_*:\mathcal G \to E(\Xmax,G)$ and if $(X,G)$ is minimal then the latter is isomorphic to $\Xmax$.
\begin{lemma}\label{lem-min}
Let $(X,G)$ be a minimal dynamical system. Then $(\pmax)_*:\mathcal G \to E(\Xmax,G)\cong \Xmax$
is an isomorphism if and only if $cr=1$ ($P=R_{max}$).
\end{lemma}
\begin{proof} We first show that $cr = 1$ implies that $(\pmax)_*$ is an isomorphism. As $(\pmax)_*$ is onto, we only have to show its injectivity.
The map $(\pmax)_*$ is injective if $(\pmax)_*(f) = \id$ implies that $f=p_0$.
Now $(\pmax)_*(f) = \id$ means that for all $\xi\in\Xmax$ and $x\in\fb{\xi }$ we have $f(x)\in\fb{\xi }$. Let $f\in p_0M$ be given such that $\pmax(x) = \pmax(f(x))$. 
By the hypothesis $cr =1$, the elements  $x$ and $f(x)$ are then proximal. Moreover, by (b) of Theorem \ref{Ellis-fundamental}, we have that the minimal ideal is unique. From part (a) of that theorem, we then infer that there exists a $p \in J$,  such that $p f(x) = p (x)$. Applying $p_0$ on both sides  and using that  $p_0 p = p_0$ by the mentioned left domination, and that $f = p_0 f$ by $f\in p_0 M$,  we then obtain
$$ f(x) = p_0 f (x) = p_0 p f (x)  = p_0 p (x)  = p_0 (x).$$
As $x$ is arbitrary, this shows  $f = p_0$.

To prove the converse suppose that $cr>1$ so that there are $\xi\in\Xmax$ and $x,x'\in\fb{\xi}$ which are not proximal. We may even assume that $p_0(x) = x$ as we can replace $x$ by $p_0 x$ and $x'$ by $p_0 x'$ and this will  not change non-proximality. (If $p_0 x $ and $p_0 x'$ were proximal, there would exist, by Theorem \ref{Ellis-fundamental}, a $q \in J$ with $q p_0 x = q p_0 x'$ and this would give $q x = q x'$ and  proximality of $x$ and $x'$ would follow from that theorem.)  By minimality of the original system and the definition of the Ellis semigroup, there exists $f\in E(X,G)$ such that $f(x') = x$. So if $m=p_0f$ we have $m(x') = x$.
Thus $(\pmax)_*(m)(\xi) = \xi$.
Since $E(\Xmax,G)$ is a group acting fixed point freely on $\Xmax$, the latter implies that  $(\pmax)_*(m) = 0$. But $m$ cannot be an idempotent, because $x'$ is not proximal to $x$.
\end{proof}
Note that
although $(\pmax)_*$ is continuous as a map from $E(X,G)$ to $E(\Xmax,G)$, one cannot conclude  in the above lemma that $\mathcal G$ is homeomorphic to $\Xmax$, as $p_0M$ need not be closed.

\smallskip

In the next section we introduce a family of Delone sets - the almost canonical cut-and-project sets - whose dynamical systems have Ellis semi-groups with a particularly nice algebraic description: the entire semi-group, not just the (unique) minimal left ideal, is isomorphic with a subgroup of the product of the maximal equicontinuous factor and a finite monoid of idempotents.


\subsection{Almost canonical cut-and-project sets}
Almost canonical cut-and-project sets are special types of complete Meyer sets. Their
internal group is a real vector space
$\R^{\np}$ where $\np$ is called the {\em codimension} of the
set. They are characterized by the form of the set, $S$, of singular
points in the maximal equicontinuous factor
$\Omax = \TM = (\R^\nd\times \R^\np)/\tilde L$; that is, the set of
points $\xi \in\TM$ which have a fiber $\fb{\xi }$ containing more
than one point.
Recall that $S$ is determined by the boundary points
$\partial W$ of the window $W$, namely
$$ S = ((\R^\nd\times\partial W) +\tilde L) / \tilde L = (\R^\nd\times
(\partial W+p_2(\tilde L))) / \tilde L.$$
\begin{definition} A complete Meyer set 
is almost canonical if 
its internal group $H$ is a vector space $\R^\np$ and its window $W$ a 
finite union of polyhedra and
the following two conditions are satisfied:
  \begin{enumerate}
  \item
There are finitely many affine hyperplanes
$A_i\subset\R^\np$, $i\in I$, such that the set $\partial W+p_2(\tilde L)$ may be
alternatively described as
$$\partial W+p_2(\tilde L) = \bigcup_{i\in I} A_i+ p_2(\tilde L).$$
\item
Any compact polyhedron whose boundary lies in $\partial W+p_2(\tilde L)$ can be constructed via a finite sequence of unions, intersections and complements of polyhedra of the form 
$W+p_2(x)$ for $x\in\tilde L$.
\end{enumerate}
\end{definition}
Such a situation arises if $W$ is a union of polyhedra whose
codimension $1$ faces span affine hyperplanes which have a
dense stabilizer under the action of $p_2(\tilde L)$. Then we may take for the $A_i$ these 
hyperplanes. 
A so-called canonical cut-and-project set is one for which $W = p_2(C)$ is
the projection of a unit cube $C$ for $\tilde L$ (w.r.t.\ to some
choice of base for $\tilde L$). It
satisfies the above criteria since the faces of the projected cube are spanned by lattice vectors.

The advantage of the alternative description of the singular points is that it leads to a
very explicit description of the hull $\Omega_\Del$. In fact
$$ \Omega_\Del = (\R^\nd\times \R_c^\np)/\tilde L$$
where $\R_c^\np$ is a locally compact totally disconnected space, a
``cut-up version'' of $\R^\np$, which is a certain completion of
$\R^\np \backslash( \bigcup_{i\in I} A_i+ p_2(\tilde L))$.
This allows the calculation of the cohomology groups (see the chapter ``Spaces of projection method patterns and their cohomology'' in this volume) and of the complexity exponents of the sets
\cite{Julien} and, as we review here, of the Ellis semigroup.

We will not directly look at the Ellis semigroup of the
(so-called continuous) dynamical system $(\Omega_\Del,\R^\nd)$ but first at
the semigroup of a reduction of it and obtain $E(\Omega_\Del,\R^N)$ by suspension.
The reduction is obtained from a choice of $\np$-dimensional subspace
$F\subset\R^\nd\times\R^\np$ whose intersection with $\tilde L$ is a rank
$\np$ subgroup $D = F\cap\tilde L$.
Then $F/D$
is  an $\np$ torus in $\TM$ which is transversal to the $\R^\nd$-action. 
The first return to $F/D$  of the equicontinuous $\R^\nd$-action on $\TM$
yields an equicontinuous $\tilde L/D$-action on $F/D$.
By construction $\tilde L/D$ is free of rank $\nd$ and so we simply
write it as $\Z^\nd$. Now let $\X=\pmax^{-1}(F/D)$. This is then transversal to the
$\R^N$-action on $\Omega_\Del$ and the first return yields an action of $\Z^\nd$ on $\X$.
It is the Ellis semigroup of this reduction $(\X,\Z^\nd)$ of $(\Omega_\Del,\R^\nd)$  which we now describe more precisely.

We denote now $\Gamma = p_2(\tilde L)$ and $\Delta = p_2(D)$. Then
$$ \X \cong \R^\np_c/\Delta$$
with $\Z^\nd$-action induced by $\Gamma$, i.e.,  $\Z^\nd = \Gamma/\Delta$.
Its maximal equicontinuous factor is $\XMmax =\Tp := \R^\np/\Delta$.

For each affine hyperplane $A_i$ there is a vector $a_i\in \R^\np$ and a
codimension $1$ subspace $H_i^0\subset \R^\np$ such that $A_i =
H_i^0+a_i$. We then must have that $\bigcap_{i\in I} H_i^0 =\{0\}$.
So the dynamical system $(\X,\Z^\nd)$ is entirely described by the data
$(\{A_i\}_{i\in I},\Gamma,\Delta)$ consisting of a finite collection of affine hyperplanes $\{A_i\}_{i\in I}$ in a real vector space $\R^\np$ such that
the intersection of their corresponding hyperspaces is trivial; a dense rank $\nd+\np$
sublattice $\Gamma$; and a rank $\np$ sublattice $\Delta$ which spans $\R^\np$.
All that follows depends only on this data and does no longer refer to a Delone set or a tiling.
The subspace $H_i^0$ cuts $\R^\np$ into two halfspaces. Choose for each $i$ a positive side. 
We denote by $H^+_i$ and 
$H^-_i$ the open half-space on the positive and the negative side, respectively. It is convenient to set $H^\infty_i=\R^\np$.
\begin{definition}
The {\em cut type} $\pt(h)$ of a point $h\in \R^\np$ is the subset
$$ \pt(h) = \{i\in I: h\in A_i+\Gamma\}.$$
A {\em point type} is a function $\pct\in \{+,-,\infty\}^I$ such that the cone
$$ C_{\pct} := \bigcap_{i\in I} H_i^{\pct(i)} $$
is non-empty. We denote by
$\mathfrak P$ the finite set of  point types.
We call $ C_{\pct}$ the (point) cone associated with $\pct$.
Its domain is
$$ \mbox{\rm dom}\, \pct =  \{i\in I: \pct(i)\neq \infty\}.$$
\end{definition}
Note that the cone $ C_{\pct}$ is a connected component of
$\R^\np \backslash \bigcup_{i\in \mbox{\rm \small dom}\, \pct} H_i^0$.
We denote by $\underline{\infty}\in \{+,-,\infty\}^I$ the function which is constant equal to $\infty$.
Its domain is empty and $C_{\underline{\infty}} = \R^\np$.
By construction $\pt(h+\gamma) = \pt(h)$ for all $\gamma\in\Gamma$  and so the cut type is also
defined for the points  of the
torus $\Tp =\R^\np/\Delta$.

Recall that the tangent cone $T_S(x)$ at $x$ of a subset $S\subset\R^\np$ is the set of vectors $v$ which can be obtained as limits of the form $v=\lim_n\frac{x_n-x}{s_n}$ where $(x_n)_n\subset S$ and $(s_n)_n\subset\R^{+}$ are sequences such that $\lim_n s_n = 0$. 
It $x$ lies in the interior of $S$ then $T_S(x) = \R^\np$.
If $C$ is a closed cone whose tip is at $0$ then $T_C(0) = C$. 
\begin{theorem}
The dynamical system $ (\X,\Z^\nd)$ is isomorphic to
$$ \X = \{(\xi ,\pct)\in\Tp\times \mathfrak P: \mbox{\rm dom}\, \pct = \pt(\xi )\}$$
with $\Z^\nd$ action given by
$t\cdot (\xi ,\pct) = (\xi  +t,\pct)$ and with topology  described in terms of convergence of sequences as
follows: 
A sequence $((\xi _n,\pct_n)_n$ converges to $(\xi,\pct)$ if and only if $\xi _n\to \xi $ in $\Tp$ and eventually $\xi _n-\xi  \in \overline{C_{\pct}}$ and
$T_{\overline{C_{\pct_n}}}(0)\subset T_{\overline{C_{\pct}}}(\xi_n-\xi)$.


The maximal equicontinuous factor of $\X$ is $\Tp$ and the factor map is the projection onto the first factor.
\end{theorem}
The expressions $\xi_n-\xi\subset \overline{C_{\pct}}$ and 
$T_{\overline{C_{\pct_n}}}(0)\subset T_{\overline{C_{\pct}}}(\xi_n-\xi)$
should be understood for
large enough $n$ so that we can lift $\xi_n-\xi$ into a small neighbourhood of $0$ in $\R^\np$ where the expressions make sense. The condition of inclusion 
$T_{\overline{C_{\pct_n}}}(0)\subset T_{\overline{C_{\pct}}}(\xi_n-\xi)$
is only relevant if 
the (lifted) sequence $\xi_n-\xi$ does not approach the tip of $C_\pct$ from inside $C_\pct$ but rather along its boundary.
This picture of the topology of $\X$ using cones is reminiscent to the oldest one, see \cite{Le},
but we refer the reader to \cite{Aujogue13} for a proof in the present framework.


\subsection{The Ellis semigroups of the systems $(\X,\Z^\nd)$ and $(\Omega_{\Del},\R^N)$}
We now consider the Ellis semigroup $E(\X,\Z^\nd)$
of the dynamical system defined by the data $(\{A_i\}_{i\in I},\Gamma,\Delta)$.
We describe its topology, its monoid structure, and finally its action on $\X$.
\begin{definition}
A transformation type is a function $\tct\in \{+,-,0\}^I$ such that
the cone
$$ C'_{\tct} := \bigcap_{i\in I} H_i^{\tct(i)} $$
is non-empty.
We denote by $\mathfrak T'$ the finite set of transformation types.
\end{definition}
Note that the cones $ C'_{\tct}$ are the
constituents of a stratification of $\R^\np$: If the $\tct(i)$ are all
different from $0$ then $C'_{\tct}$ is
a connected component of $\R^\np \backslash \bigcup_{i\in I} H_i^0$. In general $C'_{\tct}$ is  a connected component of
$\bigcap_{i:\tct(i)=0} H_i^0 \backslash \bigcup_{i:\tct(i)\neq 0} H_i^0$.
We denote by $\mathfrak o\in \{+,-,0\}^I$ the function which is constant equal to $0$. Its cone is one point: $C'_{\mathfrak o} = \{0\}$. All cones are disjoint and so only  $C'_{\mathfrak o}$ contains the origin.

Let $W_\tct$ be the connected component of
$\overline{\R C'_\tct \cap \Gamma}$ containing $0$. We call $C_\tct := W_\tct\cap C'_\tct$ the {\em effective} or {\em transformation} cone of $\tct$. It might be empty, as, for instance, if the intersection $\R C'_\tct \cap \Gamma$ is discrete but $\tct\neq\mathfrak o$.  Let
$$ \mathfrak T = \{\tct\in\mathfrak T': C_\tct \neq\emptyset\}$$
be the set of effective transformation types. For $\tct\in\Tct$ we consider
$\R^\np_\tct = \R C_\tct +\Gamma$ and its quotient $\Tp_\tct = \R^\np_\tct/\Delta$. Note that $\Gamma/\Delta \subset \Tp_\tct\subset \Tp$.
\bigskip

\noindent
{\bf Remark.} The {\em complexity} of a Delone set $\Del$ is the growth rate, as $R\to\infty$, of the number of translationally inequivalent sets of the form $B_R(x)\cap \Del$: the complexity is $\alpha$ if this number grows like $R^{\alpha}$. It is shown in \cite{Julien} that, for an almost canonical cut-and-project set, the complexity $\alpha$ satisfies $N\le\alpha\le N N^{\perp}$. While maximal complexity ($\alpha=N N^{\perp}$) is generic, many of the familiar examples - the  octagonal tilings, the Penrose tilings and their $3$-dimensional icosahedral generalisations, as well as  the Danzer tilings - have minimal ($\alpha=N$) complexity. It is a feature of almost canonical projection sets of minimal complexity that $\Gamma\cap C'_\tct$ is dense in $C'_\tct$. For those systems the effective cone $C_\tct$ coincides with $C'_\tct$, 
$W_\tct =\bigcap_{i:\tct(i) = 0} H_i^0$, and $$\Tp_\tct=\left(\bigcap_{i:\tct(i) = 0} H_i^0+\Gamma\right)/\Delta.$$
 \begin{theorem}[\cite{Aujogue13}]\label{thm-Ellis}
With the notation above, the operation
\begin{equation}\label{eq-T-prod}\nonumber
 (\mathfrak t \mathfrak t')(i) = \left\{
\begin{array}{cc}
\mathfrak t(i) & \mbox{if } \mathfrak t(i)\neq 0 \\
\mathfrak t'(i) & \mbox{if } \mathfrak t(i)= 0
\end{array}\right.
\end{equation}
defines a monoid structure\footnote{We note a difference between this formula and the one in \cite{AujoguePhD,Aujogue13} where the convention of \cite{Auslander88} that the semigroup acts from the right is used.}
on $\mathfrak T$ with $\tct = \mathfrak o$ as unit. The Ellis semigroup of $(\X,\Z^\nd)$ is isomorphic to the sub-monoid
$$E(\X,\Z^\nd) \cong \bigcup_{\tct\in\mathfrak T}  \Tp_\tct \times \{\tct\}$$
of $\Tp\times \Tct$ equipped with the product
\begin{equation}\label{eq-E-prod}\nonumber
(\xi ,\tct) (\xi ',\tct') = (\xi +\xi ',\tct \tct').
\end{equation}

Its action $E(\X,\Z^\nd)\times \X\to \X$ is given by
\begin{equation}\label{eq-E-action}\nonumber
 (\xi ,\mathfrak t)\cdot (\xi ',\mathfrak p) = (\xi +\xi ',\mathfrak p') \mbox{ where }
\:\mathfrak p'(i) = \left\{
\begin{array}{cl}
\mathfrak t(i) & \mbox{if } i\in \pt(\xi +\xi ')\mbox{ and } \tct(i)\neq 0 \\
\mathfrak p(i) & \mbox{if } i\in \pt(\xi +\xi ')\mbox{ and } \tct(i) = 0 \\
\infty & \mbox{else}
\end{array}\right. .
\end{equation}
The topology of $E(\X,\Z^\nd)$ is
first countable and may thus be described in terms of convergence of sequences.
A sequence $(\xi _n,\tct_n)_n$ in $\bigcup_{\tct\in\mathfrak T} \Tp_\tct \times \{\tct\}$
converges to $(\xi ,\tct)$ if and only if $\xi_n\to \xi$ in $\Tp$ and eventually
$\xi_n-\xi\subset {C_{\tct}}$ and
$T_{\overline{C_{\tct_n}}}(0)\subset T_{\overline{C_{\tct}}}(\xi_n-\xi)$.
\end{theorem}
Again the expressions $\xi_n-\xi\subset {C_{\tct}}$ and 
$T_{\overline{C_{\tct_n}}}(0)\subset T_{\overline{C_{\tct}}}(\xi_n-\xi)$
should be understood for
large enough $n$ so that we can lift $\xi_n-\xi$ into a small neighborhood of $0$ in $\R^\np$. 
\bigskip

\noindent
{\bf Remarks.}
The product on $\Tct$ can be described geometrically with the help of the transformation cones.
We do this below in the case of the octagonal tiling.

As it should be, the domain of $\pct'$ in the above formula is $I(\xi +\xi ')$. Indeed, if $t(i) = 0$ then $\xi \in\TM_\tct\subset H_i^0+\Gamma$ and thus $i\in I(\xi +\xi ')$ iff $i\in I(\xi ')$. This implies that
for $i\in I(\xi +\xi ')$ with $t(i)=0$ we must have $\pct(i)\neq\infty$.

Convergence of $(\xi _n,\tct_n)$ to $(\xi ,\tct)$ implies convergence of $\xi _n$ to $\xi $ in
$\Tp$. Furthermore the copy of the acting group $\Z^\nd$ in
$ E(\X,\Z^\nd)$ is given by $\alpha_t = ([t],\mathfrak o)$. 
As $T_{C_{\mathfrak o}}(0)=\{0\}$ we have that $\alpha_{t_n}$ converges to the transformation
$(\xi ,\tct)\in E(\X,\Z^\nd)$ if and only if $[t_n]\to \xi $ in $\Tp$
and eventually $[t_n]- \xi  \in {C_{\tct}}$.
\bigskip

The Ellis semigroup of the  continuous dynamical system $(\Omega_\Del,\R^\nd)$ is just the suspension of $E(\X,\Z^\nd)$,
$$E(\Omega_\Del,\R^\nd) \cong E(\X,\Z^\nd)  \times_{\Z^\nd} \R^\nd.$$
The following theorem is thus the continuous version of Theorem~\ref{thm-Ellis}.

\begin{theorem}[\cite{Aujogue13}]\label{thm-Ellis-Om}
Consider the dynamical system $(\Omega_\Del,\R^N)$ of an almost canonical cut-and-project set.
There exists a finite monoid of idempotents $\mathfrak T$ which has a unique minimal left ideal $\Tct_{min}$,
and for each $\tct\in\mathfrak T$, a group
$\TM_\tct$ with $\R^N\subset \TM_\tct\subset\TM=\Omax$ such that, algebraically,
$$ E(\Omega_\Del,\R^N) \cong \bigcup_{\tct\in\mathfrak T} \TM_\tct \times \{\tct\}\subset \TM\times\mathfrak T$$
with semigroup law
$$ (\xi ,\tct) (\xi ,\tct') = (\xi +\xi ',\tct\tct').$$
In particular the Ellis semigroup is a finite disjoint union of groups. For the unit $\mathfrak o\in\mathfrak T$ we have $\TM_{\mathfrak o} = \R^N$ and for each minimal idempotent $\tct\in\mathfrak T$ we
have $\TM_\tct = \TM$.
Finally, the semigroup morphism induced by $\pmax$ is given by the projection onto the first factor.
\end{theorem}
When we say that $ E(\Omega_\Del,\R^N)$ is a finite disjoint union of groups we mean that the semigroup law restricted to the component $\TM_\tct \times \{\tct\}$ is a group law, which follows here since
$\tct\tct=\tct$ and so $([0],\tct)$ is the neutral element in $\TM_\tct \times \{\tct\}$. But this does not mean that $ E(\Omega_\Del,\R^N)$ is a group. $\mathfrak T$ is never a group. Moreover, as is the case for Lemma~\ref{lem-min}, the identification of $E(\Omega_\Del,\R^N)$ as a submonoid of $\Omax\times\Tct$ does not respect the topology; in fact, the above theorem says nothing about the topology of $E(\Omega_\Del,\R^N)$. The local nature of the topology, and the fact that it's first countable, can be got from Theorem \ref{thm-Ellis}.

In many cases the semigroup $\mathfrak T$ is very small, containing besides $\mathfrak o$ only minimal idempotents and so $E(\Omega_\Del,\R^N) = \TM\times \Tct_{min}\cup \R^N$.
These cases constitute the generic situation in \cite{FHKmem} and correspond to cut-and-project sets with {\em maximal} complexity \cite{Julien}.
On the opposite side, the almost canonical cut-and-project sets with minimal complexity have the largest possible $\mathfrak T$ (see the remark preceding Theorem \ref{thm-Ellis}).
Less complexity seems to make the Ellis group richer!
Almost canonical cut-and-project sets with minimal complexity share also the property that their rational cohomology groups are finitely generated \cite{Julien,FHKmem}.
All Delone sets coming from primitive substitutions and, more generally, all linearly repetitive tilings, have minimal complexity. This can be rather directly  inferred from \cite{Len04} and is discussed  in some detail in   \cite{BLR}. There it  is also shown that pure point spectrum implies zero entropy (i.e., sub-exponential complexity)  for  general uniquely ergodic  systems.

\subsection{Example of the octagonal tiling}
The octagonal tiling has dimension and codimension $2$.
Its window $W$ is a regular octahedron which is the projection of the unit cube for $\tilde L = \Z^4$ onto $\R^\np=\R^2$. $\Gamma$ is the lattice generated by the difference vectors between corners of  the octahedron.
Of the eight affine hyperplanes spanned by the sides of the octahedron, only four are independent modulo $\Gamma$ so we only need four lines  $H_i^0$ and can take $a_i=0$ to describe the affine hyperplanes furnishing the input data of the dynamical system. These four lines form a
regular $8$-star in $\R^2$. We number them so that $H_1^0$ and $H_3^0$ are orthogonal and hence also $H_2^0$ and $H_4^0$.

It turns out that we have $8$ possible cut types \cite{FHKmem}: 
If $\xi\in\Gamma/\Delta$ then $I(\xi) = \{1,2,3,4\}$, i.e.\ $\xi$ lies on four different affine hyperplanes.
If  $\xi\in (\frac{e_1+e_3}2 + \Gamma)/\Delta$ then $I(\xi) = \{2,4\}$ and, if  $\xi\in (\frac{e_2+e_4}2 + \Gamma)/\Delta$ then $I(\xi) = \{1,3\}$. If $\xi\in (H_i  + \Gamma)/\Delta$ but it lies not in the above sets then the cut type is $I(\xi) = \{i\}$. Here $i=1,2,3,4$ so these yield four cut types. Finally, the cut type of all other points is $I(\xi) = \emptyset$.

It follows that we have $25$ different point types: Cut  type $\{1,2,3,4\}$ allows $8$ different point types whose cones correspond to the $8$ cones with opening angle of $45$ degree and boundary contained in $H^0_1\cup H^0_2\cup H^0_3\cup H^0_4$. Cut  type $\{1,3\}$ and $\{2,4\}$ allow each for four point types which correspond to the $4$ cones with operning angle of $90$ degree and boundary contained in $H^0_1\cup H^0_3$ and $H^0_2\cup H^0_4$, respectively.
Cut type $\{i\}$ allows for two point types corresponding to the two open half spaces bounded by $H_i^0$. Finally, if $I(\xi ) = \emptyset$ we have a single point type $\pt =  \underline{\infty}$ with cone all of $\R^2$.

An element $(\xi ,\pct)\in\X$ corresponds to a tiling.
The elements with point type $\underline{\infty}$ correspond to the non-singular tilings. The other
elements correspond to tilings which have worms. A worm\footnote{In the Penrose tiling these worms are referred to as Conway worms \cite{Grunbaum}.} 
is a configuration of tiles along a line, in the octagonal tiling made of squares and rhombi, which may occur in two different orientations. 
More precisely, three rhombi fill a hexagon, and they can do this in two different ways. In the worm one or the other way is realised coherently for all hexagons and  
changing the way is referred to as flipping the worm.
Now whenever $i\in I(\xi)$ then the tiling corresponding to $(\xi ,\pct)\in\X$ contains a worm in
a direction determined by $i$. So depending on the cut type the tiling has one, two, or four worms in one, two, or four directions, respectively, and the point type corresponds precisely to the information in which way the worms are flipped. The first coordinate $\xi $ carries the information on where the worms cross. This describes the space $\X$. The $\Z^2$-action (N=2) is given by translation of the first variable which amounts to translation of the tiling.

We now describe $E(\X,\Z^2)$.
The octagonal tiling has the property that
$C'_\tct\cap\Lambda$ is dense in $C'_\tct$ provided $\tct\neq \mathfrak o$ which implies that all cones coincide with their effective cones. In the octagonal case these cones have dimension $2$, or $1$, or $0$, the latter only for $C_\mathfrak o$. The $2$-dimensional cones coincide with the $8$ cones of the point types which have an opening angle of $45$ degree.
They are associated to the minimal idempotents and their corresponding group $\Tp_\tct$ is equal to $\Tp$.
Furthermore there are eight $1$-dimensional cones corresponding to half-lines, more precisely for each $i$ one or the other half of $H_i^0$. Their corresponding group is $(H^0_i+\Gamma)/\Delta$. Finally there is the $0$-dimensional cone $C_{\mathfrak o}=\{0\}$ whose group is $\Gamma/\Delta = \Z^2$.

The product on $\mathfrak T$ can now be described geometrically by means of the associated cones. If $C^{(2)}_1$ is a $2$-dimensional cone and $C^{(*)}_2$ any other cone we have the left domination rule $C^{(2)}_1 C^{(*)}_2=C^{(2)}_1$.
If $C^{(1)}_1$ is a $1$-dimensional cone and $C^{(2)}_2$ a two-dimensional one then
the result is the $2$-dimensional cone
$C^{(1)}_1 C^{(2)}_2=C^{(2)}_3$ which can be described as "bringing $C^{(2)}_2$ alongside":
$C^{(1)}_1$ is a half line and
$C^{(2)}_2$ an open cone which is on a distinctive side of the half line.  $C^{(2)}_3$ is then the open cone which is on the same side as $ C^{(2)}_2$ and moreover contains $C^{(1)}_1$ in its boundary.
Almost the same sort of rule applies to the product of two one-dimensional cones
$C^{(1)}_1$ and  $C^{(1)}_2$. If these correspond to half lines which are not parallel then
$C^{(1)}_1 C^{(1)}_2=C^{(2)}_3$ where $C^{(2)}_3$ is the open cone which is on the same side
of $C^{(1)}_1$ as $ C^{(2)}_1$ and contains $C^{(1)}_1$ in its boundary. If however
$C^{(1)}_1$ and  $C^{(1)}_2$ are parallel then the left domination rule applies: $C^{(1)}_1 C^{(1)}_2=C^{(1)}_1$. Finally, there is only one cone of dimension $0$, this cone corresponds to the unit.

As we have already mentionned, the transformation cones are all disjoint. But the inclusion of a cone in the (euclidean) closure of another cone has an algebraic interpretation.
Indeed, if $C_\tct \subset \overline{C_{\tct'}}$ then $\tct$ and $\tct'$ satisfy
$\tct \tct'=\tct'\tct = \tct'$ which by the general definition of the order on a idempotent semigroup
means $\tct \geq \tct'$.

The action of $E$ on $\X$ is a more complicated, as the second coordinate depends on the point type.
But it simplifies in the particular case of an idempotent as follows:
 $([0],\mathfrak t)\cdot (\xi ,\mathfrak p) =  (\xi ,\pct')$ with $\mbox{\rm dom}\, \pct' = \mbox{\rm dom}\, \pct$ and
\begin{equation}\label{eq-E-idem}\nonumber
\pct'(i)  =
 \left\{
\begin{array}{cl}
\mathfrak t(i) &  \mbox{if }\tct(i)\neq 0 \\
\mathfrak p(i) & \mbox{if }\tct(i) = 0 \end{array}\right. .
\end{equation}
This can again geometrically be described in terms of the associated cones.
The action of $\tct$ is like "bringing the cone along": the point cone $C_{\pct'}$ is the unique cone
which has the same point type as $\pct$ and contains the cone $C_\tct$ in its closure.
In particular, point type $\underline{\infty}$ is  invariant under the action of a transformation type.
The effect of the action of the idempotent $([0],\tct)$ of the Ellis semigroup on the tiling corresponding to $(\xi,\pct)$ is thus as follows: If the tiling is non-singular then it acts like the identity. If the tiling is singular then the action is to flip the worms into (or to leave them in) the position which is encoded by $\pct'$.

If the acting element $(\xi ,\tct)$ is not an idempotent, then things may become more complicated, except if $\xi \in\Gamma$ in which case the point type of $\xi +\xi '$ agrees again with that of $\xi '$ and we have the combination of a translation of the tiling with a flip of worms.
Otherwise the formula for the action given in Theorem~\ref{thm-Ellis}
 takes into account that the domain of $\pct'$ coincides with the point type of $\xi +\xi'$ and tilings may be mapped to tilings with distinct worm configurations.

\end{document}